\documentclass[preprint, 10pt,a4paper]{elsarticle}

\usepackage{amssymb,amsmath,epsfig,graphicx,amsthm, mathtools}
\usepackage{graphicx}
\usepackage{soul}
\usepackage{caption, subcaption}
\usepackage[colorlinks,allcolors=cyan!70!black,unicode]{hyperref}
\usepackage{booktabs}
\usepackage{makecell, pdflscape}
\usepackage{tabularx}
\usepackage{comment, xcolor}
\usepackage[dvipsnames]{xcolor}
\usepackage{trimclip,lipsum}
\usepackage[table]{xcolor}
\makeatletter
\def\ps@pprintTitle{%
 \let\@oddhead\@empty
 \let\@evenhead\@empty
 \def\@oddfoot{\small{ }}%
 \let\@evenfoot\@oddfoot}
\makeatother

\theoremstyle{plain}
 \newtheorem{thm}{Theorem}[section]

\theoremstyle{definition}

\theoremstyle{remark}
 
 \numberwithin{equation}{section}

\renewcommand{\leq}{\leqslant}
\renewcommand{\geq}{\geqslant}

 \DeclareMathOperator{\tr}{tr}
\begin{document}

%------Title and Author info-----------------
\title{A Model of Annual Tick Population Density in the Eastern United States as a Function of Questing Behavior and Host Availability}

%\date{February 28, 2021}

\author[uta-math]{Carli Peterson\corref{cor1}}
\ead{cxp4986@mavs.uta.edu}

\author[uta-math]{Erika Gallo}

\author[uta-math,uta-edu]{Christopher M. Kribs}

\cortext[cor1]{Corresponding author}
\address[uta-math]{Department of Mathematics, The University of Texas at Arlington, Arlington, TX 76019-0408, USA}
\address[uta-edu]
{Department of Teacher and Administrator Preparation, The University of Texas at Arlington, Arlington, TX 76019-0227, USA}

\vspace{18mm} \setcounter{page}{1} \thispagestyle{empty}

%---------Abstract--------------
\begin{abstract} %250 words or less
Understanding the ecological processes that regulate populations of the blacklegged tick, \textit{Ixodes scapularis}, is essential for explaining regional differences in tick abundance and the risk of tick-borne diseases. In this study, we develop a stage-structured nonlinear system of difference equations by ordering and composing the seasonal biological events that govern tick and host life cycles. The model incorporates ratio-dependent host-attachment, allowing host-finding success to depend jointly on host availability and tick abundance. Parameter estimates from both northeastern and southeastern ecosystems in the United States are used to examine how regional differences in host community composition and questing behavior influence tick population dynamics. Analysis shows that a $+1$/$-1$ bifurcation pair occurs at the demographic persistence threshold, producing an exchange of stability between the extinction and positive fixed points together with the emergence of an unstable cohort-based 2-cycle. Quantitative results show that ratio-dependent host-finding success limits population growth by restricting successful feeding, while questing behavior alters both the frequency and composition of host encounters. {Together, these mechanisms regulate tick persistence and long-term abundance, providing a demographic foundation for understanding variation in tick abundance and Lyme disease risk.}
\end{abstract}


\maketitle

%---------Start of main document-------------------
\section{Introduction}

\textit{Ixodes scapularis}, the blacklegged tick, has become one of the most important arthropod vectors in North America because of its expanding geographic distribution and its role in transmitting \textit{Borrelia burgdorferi}, the causative agent of Lyme disease \citep{eisen2017tick, eisen2018blacklegged, kugelersurveillance2024}. Throughout its two-year life cycle, larvae, nymphs, and adults each require a blood meal before molting or reproducing, with immature ticks feeding primarily on small vertebrates and adults feeding on medium- and large-sized mammals \citep{fish1989host, keirans1996ixodes}.

Much of the tick life cycle is spent off-host, where individuals undergo periods of questing, molting, and overwintering before locating a host \citep{opalka2011two}. During questing, ticks detect hosts using cues such as carbon dioxide, body heat, and vibrations \citep{falco1991horizontal, oorebeek2009attracts, vassallo2002comparative}. Although these mechanisms aid in host detection, only a small fraction of questing ticks ultimately obtain a blood meal \citep{ginsberg2020local}. Host finding success therefore depends on both host abundance and competition among ticks for available hosts. As host density decreases relative to tick abundance, attachment success is expected to saturate due to limited host availability and increased host defenses such as grooming \citep{ogden2005dynamic, mount1997simulation, ginsberg2020local, levin1998density, vail1997density}.

Host-finding success is also influenced by geographic variation in questing behavior. Northeastern populations of \textit{I. scapularis} typically quest above the leaf litter, whereas southern populations spend most of their time beneath the litter layer \citep{ginsberg2017environmental, arsnoe2019nymphal}. These behavioral differences are thought to reflect regional variation in environmental conditions and host communities, with reptiles comprising a larger proportion of available hosts in the southeastern U.S. \citep{ginsberg2017environmental, ostfeld2000biodiversity}. Because reptilian hosts occupy the leaf litter and are poor reservoirs for \textit{B. burgdorferi}, differences in questing behavior have been proposed as an important contributor to regional variation in both tick abundance and Lyme disease incidence \citep{ginsberg2017environmental, arsnoe2019nymphal, ginsberg2021lyme}. Consequently, understanding how host availability and questing behavior influence feeding success is essential for predicting tick population dynamics.  
%This behavioral difference complicates surveillance efforts, as standard collection methods such as flagging and dragging are ineffective for collecting juvenile ticks in southern regions \citep{ginsberg2017environmental}. As a result, data on tick population densities and activity in the southeastern U.S. are limited.

Many mathematics models have been proposed to investigate the effects of climate, host abundance, and survival on tick populations \citep{gatewood2009climate, ogden2006climate, mount1997simulation}. However, existing models typically assume constant host-finding success or represent it solely as a function of host density, making it effectively constant. In addition, continuous-time systems are commonly proposed to capture tick population dynamics, which require external seasonal forcing or time delays to accurately capture the distinct life stages ticks transition through.

{In this work, we develop a discrete-time model using a sequencing approach that constructs the system by composing biologically distinct events. In this approach, we divide the annual life cycle into distinct subintervals corresponding to reproduction, mortality, host recruitment, questing, feeding, and stage transitions. The output of each event becomes the input to the next, and composing these event-specific updates over one complete year yields a system of difference equations describing the complete annual life cycle. This method preserves the ordering of biological processes while allowing the resulting model to be analytically tractable using standard tools for discrete dynamical system.}

Using this framework, we investigate how host availability, competition for hosts, and geographic variation in questing behavior influence the population dynamics of \textit{I. scapularis}. By comparing northern and southern ecological conditions, our model provides insight into the mechanisms regulating tick persistence and establishes a modeling framework for studying seasonally structured vector populations.

\section{Model Development}
\label{development}

In this section, we develop a stage-structured discrete-time model describing the annual life cycle of \textit{Ixodes scapularis} and its sylvatic hosts. A single model framework is used for both the northeastern and southeastern U.S., with regional differences represented through changes in parameter values, including questing behavior, body burden, and host density, as well as differences in host community composition, such as the greater diversity of reptilian hosts in the Southeast \citep{ostfeld2000biodiversity, graham2010overlooked}. This approach allows us to isolate how regional variation in host communities and questing behavior influences the long-term dynamics and stability of tick populations. For modeling purposes, hosts are categorized into six groups based on similarities in their life cycles, population densities, and patterns of tick utilization. Avian hosts were not included in the present study due to limited regional density estimates and their comparatively small contribution to overall tick demography \citep{giardina2000modeling, peterson2026host}. For the northeastern parameter set, the host groups are:
\begin{itemize}
    \item Group 1: White-footed mice
    \item Group 2: Chipmunks
    \item Group 3: Short-tailed shrews, \textit{Sorex} shrews, and red-backed voles
    \item Group 4: Squirrels
    \item Group 5: Five-lined skinks and eastern fence lizards
    \item Group 6: Deer, opossums, and raccoons
\end{itemize}
and for the southeastern set:
\begin{itemize}
    \item Group 1: White-footed mice and cotton mice
    \item Group 2: Chipmunks
    \item Group 3: Short-tailed shrews and \textit{Sorex} shrews
    \item Group 4: Squirrels
    \item Group 5: Five-lined skinks, eastern fence lizards, southeastern five-lined skinks, and ground skinks
    \item Group 6: Deer, opossums, and raccoons
\end{itemize}
Section~\ref{assumptions} outlines the simplifying assumptions used in our model, and Section~\ref{model} gives a detailed description of the model structure and formulation, in which biological processes are modeled as an ordered sequence of annual events with population-specific census times. 

\subsection{Modeling Assumptions}\label{assumptions}
To model the life cycle of \textit{I. scapularis}, we use the following set of assumptions:
\begin{itemize}
    \item Larvae feed in the summer, overwinter, and emerge as nymphs the following spring. The nymphs then feed, develop into adults, and females take their final blood meal in the fall \citep{eisen2025seasonal}.
    \item Larvae and nymphs feed on all host types, while adults feed only on members of host group 6 (medium- and large-sized mammals) \citep{CDC2022}.
    \item Host deaths are concentrated into two death events per year, one in the summer and one in the winter.
    \item Host groups 1-4 reproduce twice per year (once in summer and once in the fall), while host groups 5 and 6 reproduce once per year (in the summer).
    \item Contact rates are ratio-dependent, determined by the proportion of hosts to ticks in each life stage.
    \item Host recruitment is constant, while tick reproduction depends on both the number of adults and the adult-host ratio.
    \item Larval and nymphal questing behavior is varied from north to south, while adults' is held constant \citep{tietjen2020comparative, ginsberg2021lyme}.
\end{itemize}
%{Although some overlap between life stages likely occurs in nature, particularly in the southeastern U.S. where questing seasons are longer, we assume a sequential feeding pattern in which nymphs feed in the spring, larvae in the summer, and adults in the fall. This ordering reflects the dominant seasonal structure reported in field studies \citep{eisen2025seasonal} and avoids the need for a finer temporal scale and increased model complexity.}
\newpage
\begin{landscape}
{
\renewcommand{\arraystretch}{1.6}
\begin{table}[h!]
\centering
\begin{tabular}{l|l|ccccccc}
\toprule
 & &
$t+\frac{1}{7}$ &
$t+\frac{2}{7}$ &
$t+\frac{3}{7}$ &
$t+\frac{4}{7}$ &
$t+\frac{5}{7}$ &
$t+\frac{6}{7}$ &
$t+1$ \\
\midrule

Larvae & Event
& 
&
& Larvae hatch
& Larvae die
& \cellcolor{gray!25}\begin{tabular}{c}
Larvae feed\\
(Census)
\end{tabular}
& $L\to N$
& \\

  & Eqn
& $0$
& $0$
& $\frac{r\frac{A(t-\frac{1}{7})}{2}}{\frac{\frac{A(t-\frac{1}{7})}{2}}{H_6(t-\frac{1}{7})}+K}$
& $s_L S_L\left(t+\frac{3}{7}\right)L\left(t+\frac{3}{7}\right)$
& \cellcolor{gray!25}$L\left(t+\frac{4}{7}\right)$
& $0$
& $0$ \\

\midrule

Nymphs & Event
& Nymphs die
& \cellcolor{gray!25}\begin{tabular}{c}
Nymphs feed\\
(Census)
\end{tabular}
& $N\to A$
&
&
& $L\to N$
& \\

  & Eqn
& $s_NS_N(t)N(t)$
& \cellcolor{gray!25}$N\left(t+\frac{1}{7}\right)$
& $0$
& $0$
& $0$
& $L\left(t+\frac{5}{7}\right)$
& $N\left(t+\frac{6}{7}\right)$\\

\midrule

Adults & Event
&
&
& $N\to A$
&
& \cellcolor{gray!25}\begin{tabular}{c}
Adults die\\
(Census)
\end{tabular}
& \begin{tabular}{c}Female adults feed\\and lay eggs\end{tabular}
& \\

  & Eqn
& $0$
& $0$
& $N\left(t+\frac{2}{7}\right)$
& $A\left(t+\frac{3}{7}\right)$
& \cellcolor{gray!25}$s_A S_A\left(t+\frac{4}{7}\right)A\left(t+\frac{4}{7}\right)$
& $A\left(t+\frac{5}{7}\right)$
& $0$
\\

\midrule

Host $i$ & Event
& 
& \cellcolor{gray!25}(Census)
& Hosts born
& Hosts die
& 
& Hosts born
& Hosts die\\

  & Eqn
& $H_i(t)$
& \cellcolor{gray!25}$H_i\left(t+\frac{1}{7}\right)$
& $H_i\left(t+\frac{2}{7}\right)+\frac{\Lambda_i}{2}$
& $H_i\left(t+\frac{3}{7}\right)e^{-\frac{\mu_i}{2}}$
& $H_i\left(t+\frac{4}{7}\right)$
& $H_i\left(t+\frac{5}{7}\right)+\frac{\Lambda_i}{2}$
& $H_i\left(t+\frac{6}{7}\right)e^{-\frac{\mu_i}{2}}$
\\

\midrule

Host 5 & Event
& 
& \cellcolor{gray!25}(Census)
& Hosts born
& Hosts die
& 
&
& Hosts die\\

  & Eqn
& $H_5(t)$
& \cellcolor{gray!25}$H_5\left(t+\frac{1}{7}\right)$
& $H_5\left(t+\frac{2}{7}\right)+\Lambda_5$
& $H_5\left(t+\frac{3}{7}\right)e^{-\frac{\mu_5}{2}}$
& $H_5\left(t+\frac{4}{7}\right)$
& $H_5\left(t+\frac{5}{7}\right)$
& $H_5\left(t+\frac{6}{7}\right)e^{-\frac{\mu_5}{2}}$
\\

\midrule

Host 6 & Event
&
&
& Hosts born
& Hosts die
& \cellcolor{gray!25}(Census)
&
& Hosts die\\

  & Eqn
& $H_6(t)$
& $H_6\left(t+\frac{1}{7}\right)$
& $H_6\left(t+\frac{2}{7}\right)+\Lambda_6$
& $H_6\left(t+\frac{3}{7}\right)e^{-\frac{\mu_6}{2}}$
& \cellcolor{gray!25}$H_6\left(t+\frac{4}{7}\right)$
& $H_6\left(t+\frac{5}{7}\right)$
& $H_6\left(t+\frac{6}{7}\right)e^{-\frac{\mu_6}{2}}$
\\

\bottomrule
\end{tabular}

\caption{Visual representation of how the tick–host dynamics are divided into seven sequential events. The row depicting Host $i$ represents host groups 1--4. Gray cells denote census times where the state variables are defined.}
\label{timeline}
\end{table}
}

\end{landscape}
\subsection{{The Mathematical Model}}\label{model}

Incorporating the assumptions outlined in Section~\ref{assumptions} and the sequence of events described in Table~\ref{timeline}, we develop the following system of difference equations to model the life cycle of \textit{I. scapularis} ticks and their sylvatic hosts:
\begin{equation}\label{fullsystem}
\begin{aligned}
    L(t+1)&=B(t)s_LS_L(t),\\
    N(t+1)&=s_NS_N(t)L(t),\\
    A(t+1)&=s_AS_A(t)N(t+1),\\~\\
    H_{i}(t+1)&=\left(H_{i}(t)+\frac{\Lambda_i}{2}\left(1+e^{\frac{\mu_i}{2}}\right)\right)e^{-\mu_i}, \quad \text{for $i=1,2,3,4$},\\
    H_{5}(t+1)&=\left(H_{5}(t)+\Lambda_5\right)e^{-\mu_5},\\
    H_6\left(t+1\right)&=\left(H_6(t)e^{-\frac{\mu_6}{2}}+\Lambda_6\right)e^{-\frac{\mu_6}{2}},
\end{aligned}
\end{equation}
where 
\begin{align*}
B(t) & = \frac{rA(t)/2}{\frac{A(t)/2}{H_6(t)}+K},\\
S_L(t)&=\begin{cases}
    1 - e^{- \frac{z_L}{B(t)}\Sigma_L(t)}, & A(t)>0\\
    1, & A(t)=0\end{cases},\\
    S_N(t)&=\begin{cases}
    1 - e^{- \frac{z_N}{L(t)}\Sigma_N(t)}, & L(t)>0\\
    1, & L(t)=0\end{cases},\\
    S_A(t)&=\begin{cases}
    1 - e^{- \frac{z_A}{s_NS_N(t)L(t)}\Sigma_A(t)}, & L(t)>0\\
    1, & L(t)=0\end{cases},\\
    \Sigma_L(t)&=\left(\sum_{i=1}^5 b_{i,L} \left(H_i(t) + \frac{\Lambda_i}{k_i}\right) + b_{6,L} \left(H_6(t)e^{- \mu_6/2} + \Lambda_6\right)\right),
    \text{ for } k_i=
    \begin{cases}
        2, & 1\leq i \leq 4\\
        1, & i=5
    \end{cases},
    \\
    \Sigma_N(t)&=\left(\sum_{i=1}^5 b_{i,N}\sigma_{i,N} H_i(t) + b_{6,N}\sigma_{6,N} H_6(t)e^{- \mu_6/2}\right),\\
    \Sigma_A(t)&=b_{6,A}\sigma_{6,A} H_6(t),\\
    b_{i,j} &= \hat{b}_{i,j}c \quad \text{for } i = 1,\dots,4,6; \quad b_{5,j} = \hat{b}_{5,j}(1 - c)\quad \text{and} \quad  j\in\{L,N\}.
\end{align*}

The model parameters and state variables with their descriptions and units are listed in Table~\ref{param_table}. Model~\eqref{fullsystem} incorporates nine populations: larvae ($L$), nymphs ($N$), adults ($A$), and six host groups ($H_i$, $i \in \{1,2,3,4,5,6\}$).

Due to the 2-year life cycle of \textit{I. scapularis}, $N(t+1)$ depends on $L(t)$ while $A(t+1)$ depends on $N(t+1)$. We could equivalently express the adult equation as
\begin{align*}
    A(t+1) = s_As_NS_A(t)S_N(t)L(t),
\end{align*}
which is the form we will use in the following sections and analysis.
\begin{table}[!t]
% \rotatebox[origin=c]{90}{
\centering
\resizebox{\textwidth}{!}{%
\begin{tabular}{llr}
    \toprule
    \makecell[l]{\textbf{State}\\\textbf{Variable}} & \makecell[l]{\textbf{Definition}} & \makecell[r]{\textbf{Units}}\\
    \midrule
    $L$ & Larvae & larvae/ha\\
    $N$ & Nymphs & nymphs/ha\\
    $A$ & Adults & adults/ha\\
    $H_{1}$ & White-footed mice & mice/ha\\
    $H_{2}$ & Chipmunks & chipmunks/ha\\
    $H_{3}$ & Short-tailed shrews & shrews/ha\\
    $H_{4}$ & Squirrels & squirrels/ha\\
    $H_{5}$ & Five-lined skinks & lizards/ha\\
    $H_6$ & Deer & deer/ha\\
    \midrule
    \textbf{Parameter} & \textbf{Definition} & \textbf{Units} \\
    \midrule
    $\Lambda_{i}$  & Birth/recruitment of host group $i$ & host $i$/ha\\
    $\mu_{i}$ & Proportion of host type $i$ that die each year & dimensionless\\
    $s_j$ & Natural survival of ticks in life stage $j$ & dimensionless\\
    $z_j$ & Host-finding calibration constant & dimensionless\\
    $r$ & Adults tick fecundity& larvae/host 6\\
    \makecell[l]{$K$\\\vspace{0.1cm}} & \makecell[l]{Tick-host ratio at which tick reproduction is at\\half its maximum} & \makecell[r]{adults/host 6\\\vspace{0.1cm}}\\
    \makecell[l]{$\hat{b}_{i,j}$\\\vspace{0.1cm}} & \makecell[l]{Seasonal feeding preference of tick stage $j$ for\\host group $i$}& \makecell[r]{1/host $i$\\\vspace{0.1cm}}\\
    $\sigma_{i,j}$ & Host density scaling constant & dimensionless\\
    \makecell[l]{$c$\\\vspace{0.1cm}} & \makecell[l]{Proportion of questing season spent above the\\leaf litter} & \makecell[r]{dimensionless\\\vspace{0.1cm}}\\
    \bottomrule
    \end{tabular}
    }
    \caption{State variables and parameters for Model~\eqref{fullsystem} with their descriptions and units for $i \in \{1,2,3,4,5,6\}$ and $j \in \{L,N,A\}$. }
    \label{param_table}
\end{table}

\subsubsection{Model Construction}
To derive Model~\ref{fullsystem}, following the event-sequencing approaches of \cite{porco1999mathematical} and \cite{carrera2020cost}, we construct the annual dynamics by ordering and composing biologically distinct events. Each one-year time step is divided into seven sequential subintervals representing non-overlapping processes in the tick and host life cycles. The population state produced by each event serves as the initial state for the next event, so the individual event maps are effectively stacked over the annual cycle. Their composition yields a single map from the population state at time $t$ to the population state at time $t+1$ (see Table~\ref{timeline} and~\ref{ap:timeline} for the full timeline).

Our timeline begins in the spring, when nymphs quest, feed, and drop off their hosts. Larvae hatch from eggs laid by the previous year's adults and begin questing and feeding in the summer. Of the larvae that successfully fed, some molt into nymphs before overwintering, while others overwinter as larvae and molt the following spring \citep{lindsay1998survival}. In both cases, the resulting fed larvae or newly molted nymphs remain dormant until the following spring. Nymphs that feed in the spring emerge as adults in the fall, when females take their final blood meal. These females overwinter and lay eggs the following spring, beginning the cycle anew. 

Host dynamics change along with the tick cycle. For host groups 1--4, which typically breed multiple times per year, we include two recruitment events: one in the summer (just before larval feeding) and another in the fall (just before adult feeding). Host groups 5 and 6 generally produce a single clutch or litter annually, so they experience one recruitment event in the summer during their typical reproduction season. Mortality is modeled as two events per year to approximate continuous natural losses, with one occurring immediately after recruitment to capture high juvenile mortality.

To connect these biological events to the discrete-time model, we define census points that mark the start of each time step. Because there is no point in the year when all three tick life stages are simultaneously present, it is not possible to measure every population at a single common time. Instead, we define two census points to ensure each population is captured in our model. For nymphs and host groups 1--5, census occurs in the late spring after nymphs have fed. For larvae, adults, and host group 6, the census is taken during late summer, after larvae have fed but before adults have begun questing.

Table~\ref{timeline} summarizes this construction. Each column represents one of the seven sequential events within the year, while rows correspond to the different tick life stages and host groups. The ``Event'' rows describe the biological process occurring, while the ``Eqn'' rows show the mathematical update applied to the corresponding population after that event. Gray cells indicate census points where the state variables defining the model are measured.

Beginning from the census value at time $t$, each event modifies the population according to the process listed in the table, and the resulting quantity becomes the input for the next event. After all seven events are applied, the final expression defines the population at time $t+1$. In this way, the annual difference equations in Model~\ref{fullsystem} are constructed by composing the event-specific updates over the full yearly cycle. 

This event-based construction provides a transparent and clean way to translate the seasonally structured biological processes into a discrete-time model by decomposing the annual dynamics into subintervals, then recombining them into a single yearly update.

\subsubsection{Model Parameters}
 
Host recruitment ($\Lambda_i$, host $i$/ha) and mortality ($\mu_i$, dimensionless) are treated as constants and estimated from available literature or from analysis of our model. Tick population dynamics depend on two survival probabilities: $s_j$ (dimensionless) and $S_j(t)$ (dimensionless) for $j \in \{L,N,A\}$. The quantity $(1-S_j(t))$ represents mortality due to failure to find a host, which depends on host availability and tick density and is thus a function of time. To adjust host-finding success, we introduce calibration constants $z_j$ (dimensionless). At low tick densities, when individual ticks are more likely to find a host but the expected number of ticks per host is lower overall, $z_j$ adjusts the effective encounter rate between ticks and hosts. In contrast, $(1-s_j)$ represents all other tick mortality, which primarily arises from environmental stressors (desiccation, overheating, and freezing \citep{leal2020questing, eisen2016linkages}). Thus, $s_j$ reflects survival under these pressures. Tick recruitment is determined by adult fecundity ($r$, larvae/host) and host availability, with the saturation constant $K$ (adults/host) denoting the adult tick-to-host ratio at which larval reproduction is at half its maximum.

Seasonal tick feeding preferences are represented by $\hat{b}_{i,j}$ (tick/host), defined as the total number of ticks in life stage $j$ feeding on hosts in group $i$ divided by the number of available hosts in that group. This describes each host's feeding potential and determines how ticks distribute across species. These parameters are further scaled to reflect seasonal feeding potential by accounting for the length of a questing season and the length of stage-specific tick attachment (see Section~\ref{paramestimates} for details).

For host groups that include multiple species, we chose a representative and converted other hosts into this species using estimated tick preferences ($\hat{b}_{i,j}$). Group 1 is represented by white-footed mice, group 3 by short-tailed shrews, group 5 by five-lined skinks, and group 6 by white-tailed deer. Because larval, nymphal, and adult feeding preferences differ, we include $\sigma_{i,j}$, which are scaling parameters that adjust relative host densities of group $i$ across ticks in life stage $j$.

Finally, to capture differences in questing behavior, we introduce the parameter $c$, which represents the proportion of a tick's questing season that is spent above the leaf litter. Because adults show little variation in questing behavior from north to south \citep{tietjen2020comparative, ginsberg2021lyme}, this parameter is applied only to the larval and nymphal stages. Specifically, we multiply $c$ by the seasonal tick preferences $\hat{b}_{i,j}$ for $i \in \{1,2,3,4,6\}$ and $j \in\{L,N\}$, and $(1-c)$ by $\hat{b}_{5,j}$ for $j \in \{L,N\}$, yielding the new seasonal preference parameters $b_{i,j}$. Varying $c$ between 0 and 1 allows us to examine how questing behavior influences tick dynamics. When $c \approx 1$, ticks remain above the leaf litter, reducing encounters with lizards (host group 5) ($b_{5,L}$ and $b_{5,N}$ approach 0). Conversely, when $c \approx 0$, ticks stay below the litter, increasing reptile encounters but reducing contact with other hosts ($b_{i,L}$ and $b_{i,N}$ for $i \in \{1,2,3,4,6\}$ approach 0).

\section{Model Analysis}\label{modelanalysis}
In this section, we analyze the qualitative dynamics of Model~\eqref{fullsystem}:
\begin{equation*}
\begin{aligned}
    L(t+1)&=B(t)s_LS_L(t),\\
    N(t+1)&=s_NS_N(t)L(t),\\
    A(t+1)&=s_AS_A(t)s_NS_N(t)L(t),\\~\\
    H_{i}(t+1)&=\left(H_{i}(t)+\frac{\Lambda_i}{2}\left(1+e^{\frac{\mu_i}{2}}\right)\right)e^{-\mu_i}, \quad \text{for $i=1,2,3,4$},\\
    H_{5}(t+1)&=\left(H_{5}(t)+\Lambda_5\right)e^{-\mu_5},\\
    H_6\left(t+1\right)&=\left(H_6(t)e^{-\frac{\mu_6}{2}}+\Lambda_6\right)e^{-\frac{\mu_6}{2}}.
\end{aligned}
\end{equation*}
Since nonnegative initial conditions produce nonnegative solutions, the model is biologically well posed.

The host populations evolve independently of the tick dynamics, allowing the host subsystem to be analyzed separately. Each host equation is a linear first-order difference equation of the form 
\begin{align*}
    H(t+1)=mH(t)+f,
\end{align*}
where $0<m<1$. Consequently, every host population possess a unique globally asymptotically stable fixed point. These fixed points are given by
\begin{equation*}
    \begin{aligned}
        H_{i_\infty} &= \frac{\Lambda_i}{2} \left(\frac{e^{-\mu_i}+e^{-\frac{\mu_i}{2}}}{1-e^{-\mu_i}}\right), \quad \text{for $i \in \{1,2,3,4\}$},\\
        H_{5_{\infty}}&=\frac{\Lambda_{5}e^{-\mu_5}}{1-e^{-\mu_5}},\\
        H_{6_{\infty}}&=\frac{\Lambda_{6}e^{\frac{-\mu_6}{2}}}{1-e^{-\mu_6}}.
    \end{aligned}
\end{equation*}
These fixed points exist whenever $\mu_i>0$ for $i\in\{1,\ldots,6\}$. Because the host dynamics converge independently of the tick populations, we treat the host abundances as fixed at these values for the remainder of the analysis and focus on the resulting dynamics of the tick subsystem.

We can also analyze the $LA$ subsystem separately from the $N$ equation, as the $L$ and $A$ equations do not depend on $N$, which we do in two parts. First, we examine baseline formulations that assume host-finding success is independent of tick density, including both universal host-finding success ($S_j(t)=1$) and density-independent host-finding success ($S_j(t)=\hat{S}_j$) for $j \in \{L,N,A\}$. Second, we analyze the full model, in which host-finding success depends on both tick and host densities through host-to-tick ratios.

\subsection{{Baseline Models with Constant Host-Finding Success}}\label{simple_model_analysis}

Many tick population models assume that host-finding and host-finding success are either universal or constant, removing any dependence on the tick populations themselves. To establish a baseline for comparison with the full model, we first consider two such formulations.

We begin with the assumption that every tick locates and feeds on a host so that $S_j(t)=1$ for $j \in \{L,N,A\}$. Under this assumption, survival between stages depends only on environmental survival probabilities. Once host populations have reached equilibrium, the $LA$ subsystem of Model~\eqref{fullsystem} reduces to
\begin{equation}\label{simplifiedmodel}
    \begin{aligned}
    L(t+1)&=B(t)s_L,\\
    A(t+1)&=s_As_NL(t).
    \end{aligned}
\end{equation}

Relaxing this assumption, suppose instead that host-finding success depends only on host abundance. Once host populations reach equilibrium, host-finding success becomes constant, so that $S_j(t)=\hat{S}_j$ with $0<\hat{S}_j<1$. In this case, the effective survival probability becomes $\hat{s}_j=s_j\hat{S}_j$, and the resulting system is identical to~\eqref{simplifiedmodel} after replacing $s_j$ with $\hat{s}_j$. Consequently, both forms of the model produce the same qualitative behavior, which we discuss next. Section \ref{populationdynamics_full} investigates the differences quantitatively. 

Model~\eqref{simplifiedmodel} has two fixed points:
\begin{enumerate}
    \item Extinction fixed point: $(L_0,A_0) = (0,0)$,
    \item Existence fixed point: $(L_\infty,A_\infty) = \left(\frac{2H_{6_{\infty}}(r{S}-K)}{{s}_N{s}_A}, 2H_{6_{\infty}}(r{S}-K)\right)$,
\end{enumerate}
where ${S}={s}_L{s}_N{s}_A/2$. The coexistence fixed point exists only when $\mathcal{R}_d=\frac{rS}{K}>1$, where $\mathcal{R}_d$ is the demographic reproductive number of the baseline model. This threshold is obtained by linearizing the system about the extinction fixed point and computing the spectral radius of the associated Jacobian matrix \citep{cushing1994net, cushing1989ebenman}.

The quantity $\mathcal{R}_d$ separates extinction from persistence. When $\mathcal{R}_d<1$, mortality exceeds replacement and the extinction fixed point is locally asymptotically stable (\ref{ap:qual_analysis}). Once $\mathcal{R}_d>1$, a positive fixed point emerges and becomes locally stable.

We next examine the second-generation system: 
\begin{equation*}
    \begin{aligned}
    L(t+2)&=\frac{r\frac{{s}_N{s}_AL(t)}{2}}{\frac{\frac{{s}_N{s}_AL(t)}{2}}{H_{6_{\infty}}}+K}{s}_L,\\
    A(t+2)&={s}_A{s}_NB(t){s}_L.
    \end{aligned}
\end{equation*}
In addition to the fixed points of the original system, the second-generation system admits the solution
\begin{align*}
    \left(L_{1_\infty},A_{1_\infty}\right)&=\left(\frac{2H_{6_{\infty}}(r{S}-K)}{{s}_N{s}_A},0\right), \quad \text{and} \quad \left(L_{2_\infty},A_{2_\infty}\right)=\left(0, 2H_{6_{\infty}}(r{S}-K)\right),
\end{align*}
whenever $\mathcal{R}_d>1$. Biologically, these points correspond to a synchronous 2-cycle in which one of the two overlapping annual cohorts is lost, leaving the reaming cohort to alternate between larval and adult stages over successive years. However, \ref{ap:second_gen_analysis} shows that this two-cycle is never stable.

\begin{figure}[!hp]
    \centering
    \includegraphics[width=0.5\linewidth]{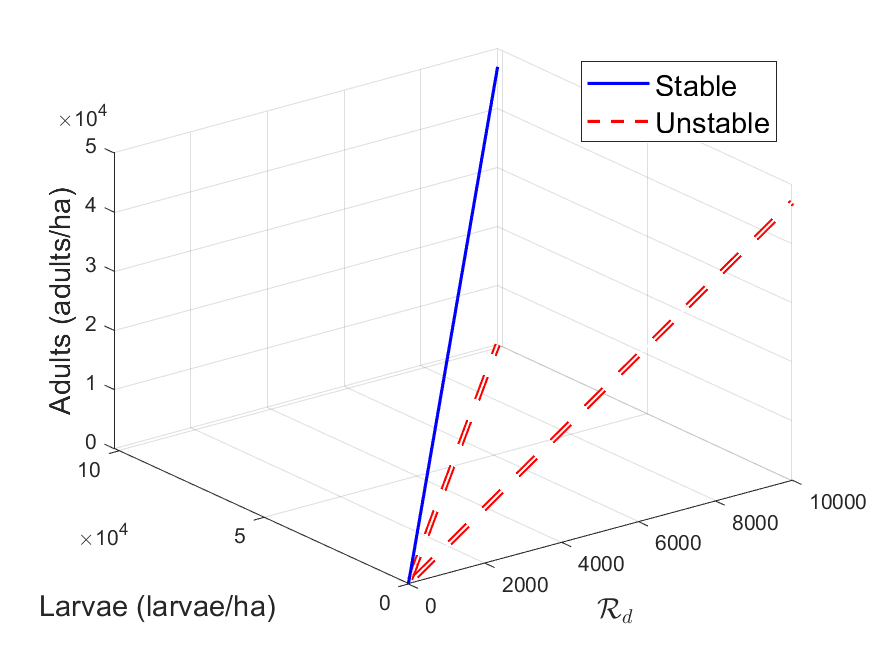}
    \caption{Three-dimensional bifurcation diagram for Model~\eqref{simplifiedmodel} under universal feeding showing larval and adult densities with respect to $\mathcal{R}_d$. Simulations use the northern parameter set from Table~\ref{paramvaluesnorthsouth}, with all parameters fixed except $r$, which was varied to change $\mathcal{R}_d$. The extinction and existence fixed points are shown as single lines, solid blue when stable and dashed red when unstable, while the 2-cycle is shown as a double line with the same color and line-style convention. The baseline model with fixed host-finding success ($S_j=\hat{S}_j$) exhibits the same qualitative structure and differs only in scale.}
    \label{bifurcation_simple}
\end{figure}

Figure~\ref{bifurcation_simple} summarizes the qualitative dynamics of the baseline model under universal host-finding success as the demographic reproductive number, $\mathcal{R}_d$, varies. Stable fixed points are shown as solid blue curves, unstable fixed points as dashed red curves, and the unstable synchronous 2-cycle as red double curves. When host-finding success is assumed constant ($S_j=\hat{S}_j$), the same bifurcation structure is obtained, with only the population sizes changing. 

At the threshold $\mathcal{R}_d=1$, the system simultaneously undergoes a transcritical bifurcation and a period-doubling bifurcation \citep{cushing1989ebenman}. The transcritical bifurcation occurs as an eigenvalue passes through $+1$, causing the extinction fixed point to lose stability while the positive fixed point becomes stable. At the same parameter value, the second iterate experiences a $-1$ bifurcation, generating a synchronous period-two solution. 

This 2-cycle reflects the two-year life cycle of \textit{I. scapularis}. Under normal conditions, two overlapping cohorts are present in the population. Along the synchronous 2-cycle, one of these cohorts is lost, leaving a single cohort that alternates between the larval and adult stages in successive years. From a mathematical perspective, this behavior follows from the Leslie projection matrix associated with the system, which is irreducible and imprimitive with index of imprimitivity two. The resulting periodic structure gives rise to the synchronous 2-cycle \citep{allen1989density}.

{Deriving the fourth-generation system and analyzing it shows that there are no additional solutions beyond those of the first- and second-generation systems. Therefore, there are no higher-order cycles, and the positive fixed point is globally stable whenever it is locally stable, except for trajectories initialized exactly on the unstable 2-cycle.}

\subsection{{Full Model}}
The nonlinearity introduced in the full system (Model~\eqref{fullsystem}) prevents us from obtaining a closed-form expression for the positive fixed point. Nevertheless, some qualitative properties can be established analytically. After substituting the expression for $S_j(t)$ from Section~\ref{model}, the tick subsystem becomes
\begin{equation*}
    \begin{aligned}
    L(t+1)&=B(t)s_L\left(1 - e^{- \frac{z_L}{B(t)}\Sigma_L}\right),\\
    A(t+1)&=s_A\left(1 - e^{- \frac{z_A}{s_NS_N(t)L(t)}\Sigma_A}\right)s_N\left(1 - e^{- \frac{z_N}{L(t)}\Sigma_N}\right)L(t),
    \end{aligned}
\end{equation*}
for $L,A>0$.

The model always admits an extinction fixed point $(L,A)=(0,0)$, while a unique positive fixed point exists when $\mathcal{R}_d=rS/K>1$, with $S=s_Ls_Ns_A/2$ (see~\ref{ap:existence} for details).

Linearization about the extinction fixed point shows that its stability is determined by the same threshold as in the baseline model. Specifically, the extinction fixed point is locally asymptotically stable whenever $\mathcal{R}_d<1$ and loses stability at $\mathcal{R}_d=1$ through simultaneous $+1$ and $-1$ bifurcations (\ref{ap:full_qual_analysis}). Similarly, the positive fixed point is stable for $\mathcal{R}_d>1$ (\ref{ap:full_qual_analysis}).

To determine whether the full system supports oscillatory dynamics, we next examine the second-generation system. This system, given in \ref{ap:second_gen_full}, contains two additional fixed points corresponding to the endpoints of the same synchronous 2-cycle found in the baseline model. 

\section{Quantitative Results}
\label{simulations}
In this section, we discuss parameter estimates used to parameterize the model, detail our numerical analysis, and discuss the results. 

\subsection{Parameter Estimates}\label{paramestimates}

Model parameters were estimated separately for the northeastern and southeastern U.S. using region-specific data. The resulting parameter values are summarized in Table~\ref{paramvaluesnorthsouth}.

Host community composition was estimated from published surveys conducted within each region. When several species were grouped into a single host class, observed densities were converted to an equivalent representative species as described in Section~\ref{model}. Additional details on the northern dataset are given in \cite{peterson2026host}, while the southern estimates are summarized in Table~\ref{southernhostdensities}.

Because drag and flag sampling methods typically underestimate tick abundance, we based our baseline tick densities on mark-recapture measurements \citep{daniels2000estimating}, which reported densities of 115,000 larvae/ha, 11,500 nymphs/ha, and 3,450 adults/ha. Due to a lack of data for southern regions, particularly because common sampling methods are not effective in collecting juvenile ticks in those regions \citep{ginsberg2021lyme}, we assume the same baseline densities in the south.

Off-host survival proportions were taken from \cite{brunner2023off} and assumed to be identical across regions. The calibration constants $z_j$ were then chosen so that, under the baseline host community, the model reproduced the observed overall survival proportions of approximately 10\% for larvae, 10\% for nymphs, and 30\% for adults \citep{daniels2000estimating, ostfeld2023mouse}. 

Annual larval recruitment was determined by two parameters: the reproductive parameter $r$ and the saturation constant $K$. We estimated $r$ by combining published values for female fecundity, egg survival, and the average number of adult females feeding on an individual deer during a questing season \cite{piesman1979host, baker2024feeding}. This yielded an estimate of 2.52 million larvae/deer. The saturation parameter $K$ was then selected so that $\mathcal{R}_d>1$, ensuring persistence of the baseline population so that the effects of questing behavior on the tick population could be explored.

Host demographic parameters were estimated from published data. Recruitment was computed from reported reproductive output under the assumption of an equal sex ratio and was assumed to be the same in both the northern and southern regions. Mortality rates were chosen to satisfy the fixed-point conditions derived in Section~\ref{modelanalysis}. 

Stage-specific host preferences were estimated primarily from the field data of \cite{ginsberg2021lyme}. Northern preferences were based on observations from Massachusetts, New Jersey, and Wisconsin, whereas southern values used data from North Carolina, South Carolina, and Tennessee. Because deer were not included in these data, deer preferences were obtained from \cite{piesman1979role}. Preference estimates were then used to convert host abundances to representative-species densities, with scaling parameters $\sigma_{i,j}$ accounting for differences among tick life stages.

Because estimates of larval and nymphal body burden are scarce for many southern hosts, when direct measurements were unavailable, southern burdens were assumed to be 10\% of northern values to reflect lower attachment rates to mammalian hosts. Adult burdens were taken to be identical across regions because adult questing behavior is similar throughout the eastern U.S. \cite{arsnoe2019nymphal, ginsberg2021lyme}.

The baseline questing parameter $c$ was estimated from the observations of \cite{arsnoe2019nymphal}, who measured the proportion of time nymphs spent questing above the leaf litter in northern and southern field sites. Average values of 18.15\% and 2.04\% were used for the northern and southern parameter sets, respectively. Because comparable observations are unavailable for larvae, we assumed the same regional values for both immature life stages. 

\begin{table}[htbp!]
\centering
\resizebox{\textwidth}{!}{%
\begin{tabular}{llrrr}
\toprule
& & \multicolumn{2}{c}{\textbf{Value [Source]}} & \\
\cmidrule(lr){3-4}
\makecell{\textbf{Baseline}\\\textbf{Densities}} & \textbf{Description} & \textbf{North} & \textbf{South} & \textbf{Units} \\
\midrule
$L(0)$ & Larvae
    & $115{,}000$ \citep{daniels2000estimating}
    & $115{,}000$ \citep{daniels2000estimating}
    & larvae/ha \\
$N(0)$ & Nymphs
    & $11{,}500$ \citep{daniels2000estimating}
    & $11{,}500$ \citep{daniels2000estimating}
    & nymph/ha \\
$A(0)$ & Adults
    & $3{,}450$ \citep{daniels2000estimating}
    & $3{,}450$ \citep{daniels2000estimating}
    & adult/ha \\
$H_{1}(0)$ & Mice
    & $40$ \citep{peterson2026host}
    & $27.9$
    & mice/ha \\
$H_{2}(0)$ & Chipmunks
    & $30$ \citep{peterson2026host}
    & $30$
    & chipmunks/ha \\
$H_{3}(0)$ & Shrews
    & $59.2$ \citep{peterson2026host}
    & $30$
    & shrews/ha \\
$H_{4}(0)$ & Squirrels
    & $6$ \citep{peterson2026host}
    & $6$
    & squirrels/ha \\
$H_{5}(0)$ & Lizards
    & $27$ \citep{peterson2026host}
    & $445.6$
    & lizard/ha \\
$H_{6}(0)$ & Deer
    & $0.46$ \citep{peterson2026host}
    & $0.40$
    & deer/ha \\

\midrule
& & \multicolumn{2}{c}{\textbf{Value [Source]}} & \\
\cmidrule(lr){3-4}
\textbf{Parameters} & \textbf{Description} & \textbf{North} & \textbf{South} & \textbf{Units} \\
\midrule
$\Lambda_1$ & Recruitment of mice
    & $300^{a}$
    & $273.5^{a}$
    & mice/ha \\
$\Lambda_2$ & Recruitment of chipmunks
    & $50$ \citep{snyder1982tamias}
    & $50$ \citep{snyder1982tamias}
    & chipmunks/ha \\
$\Lambda_3$ & Recruitment of shrews
    & $659.8^{b}$
    & $174.9^{b}$
    & shrews/ha \\
$\Lambda_4$ & Recruitment of squirrels
    & $27^{c}$
    & $27^{c}$
    & squirrels/ha \\
$\Lambda_5$ & Recruitment of lizards
    & $36.1^{d}$
    & $1278.1^{d}$
    & lizards/ha \\
$\Lambda_6$ & Recruitment of deer
    & $0.32^{e}$
    & $0.21^{e}$
    & deer/ha \\

$s_L$ & Natural survival of larvae
    & $0.43$ \citep{brunner2023off}
    & $0.43$ \citep{brunner2023off}
    & dimensionless \\
$s_N$ & Natural survival of nymphs
    & $0.55$ \citep{brunner2023off}
    & $0.55$ \citep{brunner2023off}
    & dimensionless \\
$s_A$ & Natural survival of adults
    & $0.84$ \citep{brunner2023off}
    & $0.84$ \citep{brunner2023off}
    & dimensionless \\

$\mu_1$ & Natural death of mice
    & $3.12$
    & $3.55$
    & dimensionless \\
$\mu_2$ & Natural death of chipmunks
    & $1.21$
    & $1.21$
    & dimensionless \\
$\mu_3$ & Natural death of shrews
    & $3.77$
    & $2.73$
    & dimensionless \\
$\mu_4$ & Natural death of squirrels
    & $2.36$
    & $2.36$
    & dimensionless \\
$\mu_5$ & Natural death of lizards
    & $0.85$
    & $1.35$
    & dimensionless \\
$\mu_6$ & Natural death of deer
    & $0.68$
    & $0.52$
    & dimensionless \\

$\hat{b}_{1,L}$ & Relative proportion of larvae feeding on mice
    & $404.4$ \citep{ginsberg2021lyme}
    & $152.0$ \citep{ginsberg2021lyme}
    & larvae/mouse \\
$\hat{b}_{2,L}$ & Relative proportion of larvae feeding on chipmunks
    & $245.2$ \citep{ginsberg2021lyme}
    & $220.6$
    & larvae/chipmunk \\
$\hat{b}_{3,L}$ & Relative proportion of larvae feeding on shrews
    & $483.7$ \citep{ginsberg2021lyme}
    & $1{,}693.6$ \citep{ginsberg2021lyme}
    & larvae/shrew \\
$\hat{b}_{4,L}$ & Relative proportion of larvae feeding on squirrels
    & $911.3$ \citep{main1982immature}
    & $810.8$
    & larvae/squirrel \\
$\hat{b}_{5,L}$ & Relative proportion of larvae feeding on lizards
    & $35.3$ \citep{ginsberg2021lyme}
    & $25.7$ \citep{ginsberg2021lyme}
    & larvae/lizard \\
$\hat{b}_{6,L}$ & Relative proportion of larvae feeding on deer
    & $1{,}0817.1$ \citep{piesman1979role}
    & $962.3$
    & larvae/host 6 \\

$\hat{b}_{1,N}$ & Relative proportion of nymphs feeding on mice
    & $20.4$ \citep{ginsberg2021lyme}
    & $15.2$ \citep{ginsberg2021lyme}
    & nymphs/mouse \\
$\hat{b}_{2,N}$ & Relative proportion of nymphs feeding on chipmunks
    & $188.4$ \citep{ginsberg2021lyme}
    & $167.6$
    & nymphs/chipmunk \\
$\hat{b}_{3,N}$ & Relative proportion of nymphs feeding on shrews
    & $1.9$ \citep{ginsberg2021lyme}
    & $1.4$
    & nymphs/shrew \\
$\hat{b}_{4,N}$ & Relative proportion of nymphs feeding on squirrels
    & $1{,}527.8$ \citep{main1982immature}
    & $1{,}359.3$
    & nymphs/squirrel \\
$\hat{b}_{5,N}$ & Relative proportion of nymphs feeding on lizards
    & $0.4$ \citep{ginsberg2021lyme}
    & $3.0$ \citep{ginsberg2021lyme}
    & nymphs/lizard \\
$\hat{b}_{6,N}$ & Relative proportion of nymphs feeding on deer
    & $4{,}132.2$ \citep{piesman1979role}
    & $3{,}676.5$
    & nymphs/deer \\

$b_{6,A}$ & Relative proportion of adults feeding on deer
    & $1{,}037.5$ \citep{piesman1979role}
    & $1{,}037.5$ \citep{piesman1979role}
    & adults/deer \\

\makecell[l]{$\sigma_{1,N}$\\\vspace{0.1cm}}
& \makecell[l]{Scaling constant for preference of nymphs\\
(versus larval) for host group 1}
& ---
& \makecell[r]{$1.24$\\\vspace{0.1cm}}
& \makecell[r]{dimensionless\\\vspace{0.1cm}} \\

\makecell[l]{$\sigma_{3,N}$\\\vspace{0.1cm}}
& \makecell[l]{Scaling constant for preference of nymphs\\
(versus larval) for host group 3}
& \makecell[r]{$3.59$\\\vspace{0.1cm}}
& \makecell[r]{$0.90$\\\vspace{0.1cm}}
& \makecell[r]{dimensionless\\\vspace{0.1cm}} \\

\makecell[l]{$\sigma_{5,N}$\\\vspace{0.1cm}}
& \makecell[l]{Scaling constant for preference of nymphs\\
(versus larval) for host group 5}
& \makecell[r]{$1.07$\\\vspace{0.1cm}}
& \makecell[r]{$1.19$\\\vspace{0.1cm}}
& \makecell[r]{dimensionless\\\vspace{0.1cm}} \\

\makecell[l]{$\sigma_{6,N}$\\\vspace{0.1cm}}
& \makecell[l]{Scaling constant for preference of nymphs\\
(versus larval) for host group 6}
& \makecell[r]{$1.25$\\\vspace{0.1cm}}
& \makecell[r]{$1.00$\\\vspace{0.1cm}}
& \makecell[r]{dimensionless\\\vspace{0.1cm}} \\

\makecell[l]{$\sigma_{6,A}$\\\vspace{0.1cm}}
& \makecell[l]{Scaling constant for preference of adults\\
(versus larval) for host group 6}
& \makecell[r]{$1.08$\\\vspace{0.1cm}}
& \makecell[r]{$1.00$\\\vspace{0.1cm}}
& \makecell[r]{dimensionless\\\vspace{0.1cm}} \\

$r$ & Fecundity of adult ticks
    & $2{,}520{,}000$ \citep{baker2024feeding, piesman1979role}
    & $2{,}520{,}000$
    & larvae/host 6 \\

\makecell[l]{$K$\\\vspace{0.1cm}}
& \makecell[l]{Tick-to-host ratio at which tick reproduction is at\\
half its maximum}
& \makecell[r]{$5$\\\vspace{0.1cm}}
& \makecell[r]{$5$\\\vspace{0.1cm}}
& \makecell[r]{adult/host 6\\\vspace{0.1cm}} \\

$z_L$ & Larval host-finding success calibration constant
    & $1.24$
    & $1.45$
    & dimensionless \\
$z_N$ & Nymphal host-finding success calibration constant
    & $1.30$
    & $2.58$
    & dimensionless \\
$z_A$ & Adult host-finding success calibration constant
    & $1.79$
    & $2.23$
    & dimensionless \\
$c$ & Proportion of questing season spent above the leaf litter
    & $0.18$ \citep{arsnoe2019nymphal}
    & $0.02$ \citep{arsnoe2019nymphal}
    & dimensionless \\

\bottomrule
\end{tabular}
}

\caption{Baseline densities and parameters used for the northeast (North) and southeast (South) parameterizations of Model~\eqref{fullsystem}, with their descriptions, estimated values, and units. Parameters without cited sources were calculated as described in Section~\ref{paramestimates} and Appendix~\ref{ap:paramcalc}. Host densities for the South parameterization were estimated from literature (see Table~\ref{southernhostdensities}).\\
$^{a}$\citep{morris1986proximate, mehlhop1978population, chitty1941territorial}\\
$^{b}$\citep{innes1990high, whitaker1998mammals, merritt1981clethrionomys, innes1984life}\\
$^{c}$\citep{layne1952genitale, barkalow1970vital, kemp1970dynamics}\\
$^{d}$\citep{COSEWIC2007, seburn1990population, tinkle1972sceloporus}\\
$^{e}$\citep{verme1984physiology, lay1942ecology, johnson2016habitat}}
\label{paramvaluesnorthsouth}

\end{table}

Numerical solutions to Model~\eqref{fullsystem} were obtained in MATLAB\textsuperscript{\textregistered} using \verb|fsolve| and are shown in the next section. To identify biologically relevant equilibria, the solver was initialized using multiple starting guesses over feasible ranges for tick and host populations. Stability of each solution was assessed using eigenvalues of numerically computed Jacobian matrices evaluated at the corresponding fixed point. A similar procedure was followed using the second-generation system to find 2-cycles and their stability.

\subsection{{Impact of Feeding Assumptions on Qualitative Dynamics}}\label{populationdynamics_full}

\begin{figure}[!hp]
%\centerline{
\centering
\begin{subfigure}[t]{0.5\textwidth}
    \centering
    \includegraphics[width=\linewidth]{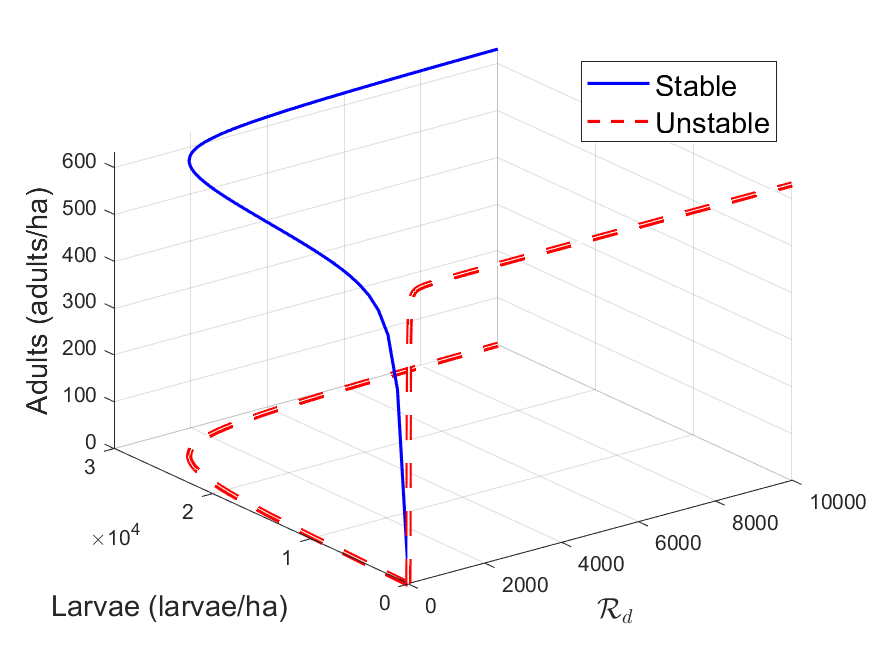}
    \caption{}
    \label{fig:2a}
\end{subfigure}%
\begin{subfigure}[t]{0.5\textwidth}
    \centering
    \includegraphics[width=\linewidth]{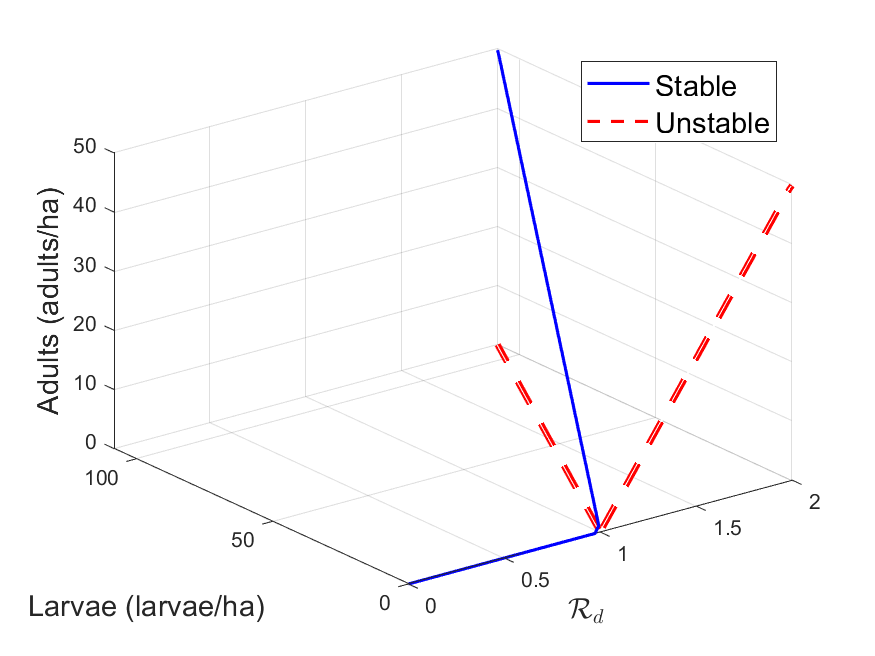}
    \caption{}
    \label{fig:2b}
\end{subfigure}%
\caption{Bifurcation diagrams for Model~\eqref{fullsystem} of the larval and adult populations with respect to the demographic reproductive value $\mathcal{R}_d=rS/K$ using the northern parameter set from Table~\ref{paramvaluesnorthsouth}, with all parameters fixed except $r$, which was varied to change $\mathcal{R}_d$. Fixed points are shown as a single line, solid blue when stable and dashed red when unstable, while the 2-cycle is shown as a double line with the same color and line-style convention. Subfigure~\ref{fig:2b} provides a zoomed view of subfigure~\ref{fig:2a}, highlighting the bifurcations at $\mathcal{R}_d=1$.}
\label{bifurcation_diagrams_rS}
\end{figure}

Figure~\ref{bifurcation_diagrams_rS} shows bifurcation diagrams for Model~\eqref{fullsystem} with respect to the demographic reproductive value $\mathcal{R}_d=rS/K$, where $S=s_Ls_Ns_A/2$, using the northern parameter set given in Table~\ref{paramvaluesnorthsouth}. Fixed points are shown as a single line (solid blue when stable, dashed red when unstable), while 2-cycle solutions are shown as double lines with the same stability convention. Subfigure~\ref{fig:2b} displays a zoomed view of subfigure~\ref{fig:2a} to highlight the bifurcations that occur at $\mathcal{R}_d=1$.

As in the baseline model, the threshold $\mathcal{R}_d=1$ corresponds to a $+1/-1$ double bifurcation. At this point, a transcritical bifurcation exchanges stability between the extinction fixed point and the positive existence fixed point, while a simultaneous period-doubling bifurcation generates an unstable 2-cycle.

For larger values of $\mathcal{R}_d$, the three model formulations exhibit distinct quantitative behavior. Under the constant (fixed) and universal host-finding success assumptions, fixed-point tick densities increase approximately linearly with $\mathcal{R}_d$. In the case of universal host-finding success, the model predicts substantially higher fixed-point densities than the ratio-dependent case. In contrast, densities under the ratio-dependent feeding model increase at a saturating rate and approach a limiting value as $\mathcal{R}_d$ becomes large. These results demonstrate that the simplifying assumptions preserve the qualitative dynamics of the system but alter fixed-point population sizes.  

While Figures~\ref{bifurcation_simple} and~\ref{bifurcation_diagrams_rS}  demonstrate that the three model formulations exhibit similar qualitative dynamics, the magnitude of the quantitative differences warrants further investigation.

\subsection{{Population Effects of Ratio-Dependent Host-Finding Success}}

\begin{figure}[!hp]
%\centerline{
\centering
\begin{subfigure}[t]{0.5\textwidth}
    \centering
    \includegraphics[width=\linewidth]{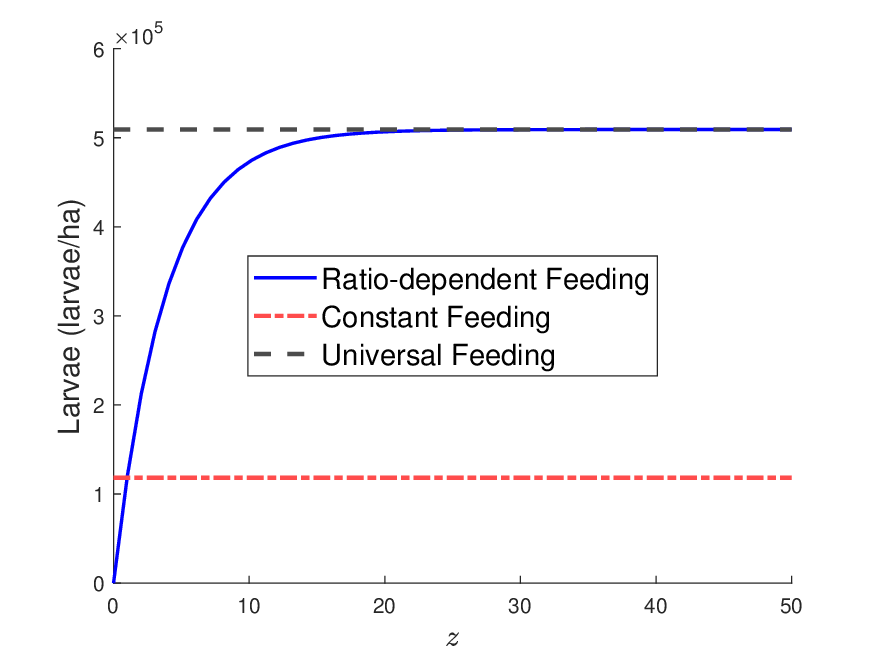}
    \caption{}
    \label{fig:3a}
\end{subfigure}%
\begin{subfigure}[t]{0.5\textwidth}
    \centering
    \includegraphics[width=\linewidth]{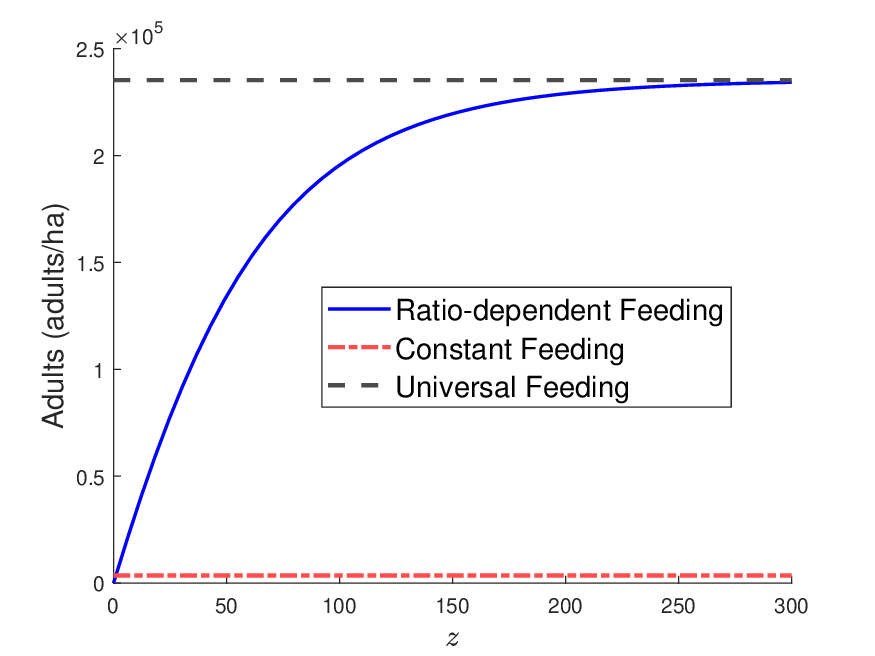}
    \caption{}
    \label{fig:3b}
\end{subfigure}%
\caption{Fixed-point dynamics of the larval and adult populations under the constant host-finding success assumption (dash-dotted red line), the universal host-finding success assumption (dashed black line) and the ratio-dependent host-finding success assumption (solid blue line). Population densities are with respect to $z$, which was used to uniformly scale the host-finding success calibration constant $z_j$ across all tick stages. Results shown use parameter estimates from the northeastern U.S. (Table~\ref{paramvaluesnorthsouth}).}
\label{bifurcation_diagrams_z}
\end{figure}
Figure~\ref{bifurcation_diagrams_z} compares fixed-point larval and adult densities under the constant (fixed) host-finding success, universal host-finding success, and ratio-dependent host-finding success assumptions as the host-finding calibration parameters $z_j$ are scaled uniformly by $z$. Increasing $z$ increases host-finding success and therefore provides a means of examining the effect of the ratio-dependent feeding term on fixed-point population densities. 

Across both life stages, fixed-point densities predicted by the ratio-dependent formulation increase monotonically with $z$. The universal host-finding success model is obtained as the limiting case as $z_j \to \infty$. Consequently, for large values of $z$, the ratio-dependent formulation asymptotically approaches the universal host-finding success case. In contrast, the constant host-finding success model assumes the product $s_jS_j(t)$ remains fixed at its calibrated baseline value, producing a single fixed-point value that is independent of $z$. For all tick stages, both the constant and universal host-finding success models predict higher fixed-point densities at the calibrated $z_j$ values (corresponding to $z=1$) than the ratio-dependent form.

Thus, although the simplified feeding assumptions preserve the qualitative structure of the system, they produce different quantitative results.

\subsection{{Population Effects of Questing Behavior}}\label{stability_fp}

Figure~\ref{bifurcation_diagrams_c} presents bifurcation diagrams with respect to the questing behavior parameter $c$ for the northern (left columns) and southern (right column) parameter sets given in Table~\ref{paramvaluesnorthsouth}. Because the chosen parameter values satisfy $\mathcal{R}_d>1$, the bifurcations seen in Figure~\ref{bifurcation_diagrams_rS} are not displayed in these figures. Instead, Figure~\ref{bifurcation_diagrams_c} illustrates how variation in questing behavior influences the size of the tick population.

For both parameter sets, fixed-point tick densities increase monotonically with $c$, reflecting improved host-finding success. The magnitude of this increase, however, differs between regions. For very low values of $c$, northern populations approach extinction, whereas southern populations remain above these levels. Interestingly, southern populations exceed those in the north over a portion of the range of $c$. For $c$ larger than approximately 0.83, northern population densities eventually exceed those of the southern system.

Differences between life stages are also evident from the figures. Adult densities approach their limit more rapidly than larval densities, while larval populations continue increasing almost linearly over the range of $c$ considered. 

\begin{figure}[!hp]
%\centerline{
\centering
\begin{subfigure}[t]{0.5\textwidth}
    \centering
    \includegraphics[width=\linewidth]{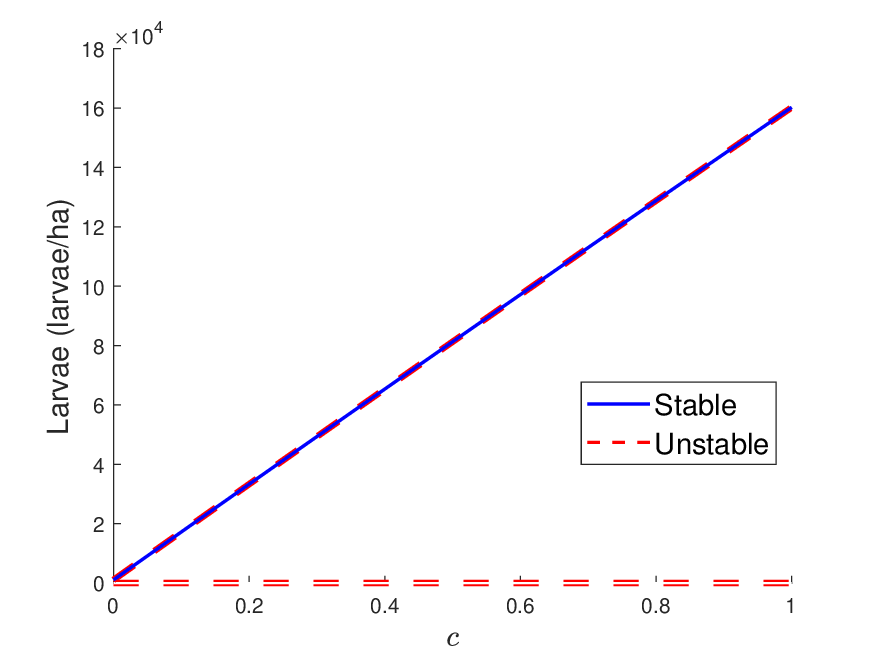}
    \caption{}
    \label{fig:4a}
\end{subfigure}%
\begin{subfigure}[t]{0.5\textwidth}
    \centering
    \includegraphics[width=\linewidth]{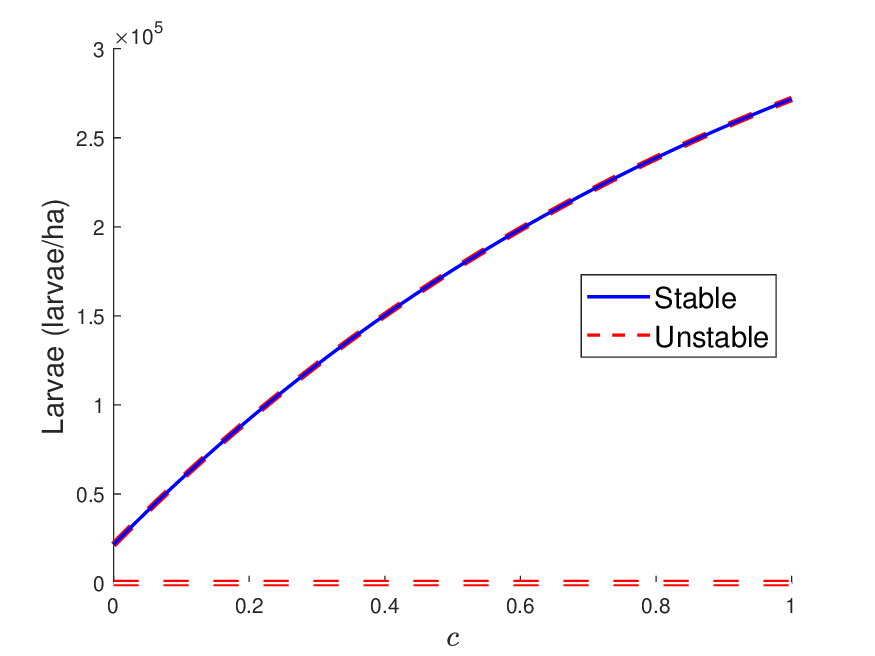}
    \caption{}
    \label{fig:4b}
\end{subfigure}%
\hfill
\begin{subfigure}[t]{0.5\textwidth}
    \centering
    \includegraphics[width=\linewidth]{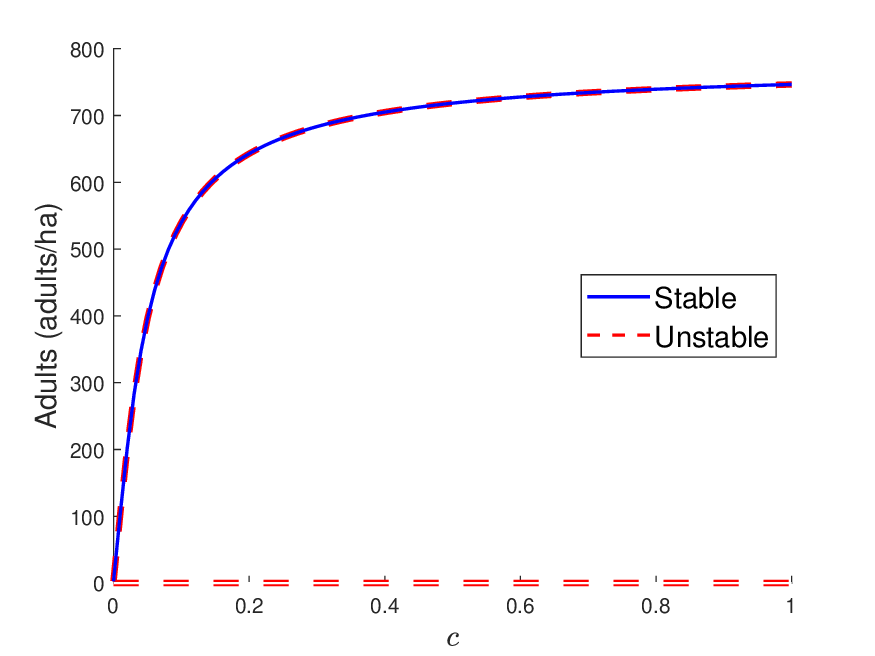}
    \caption{}
    \label{fig:4c}
\end{subfigure}%
%\hfill
\begin{subfigure}[t]{0.5\textwidth}
    \centering
    \includegraphics[width=\linewidth]{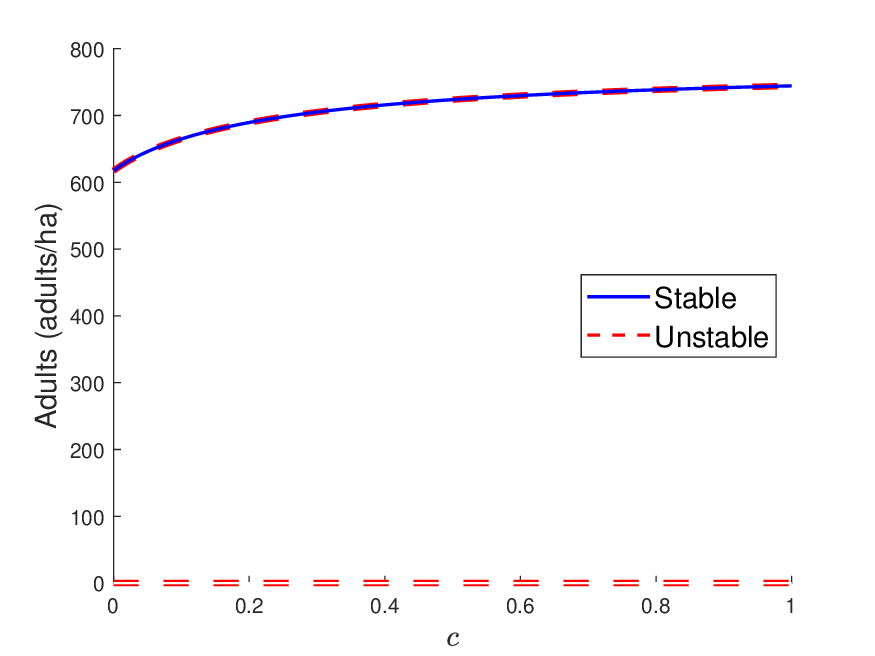}
    \caption{}
    \label{fig:4d}
\end{subfigure}%
\hfill
\begin{subfigure}[t]{0.5\textwidth}
    \centering
    \includegraphics[width=\linewidth]{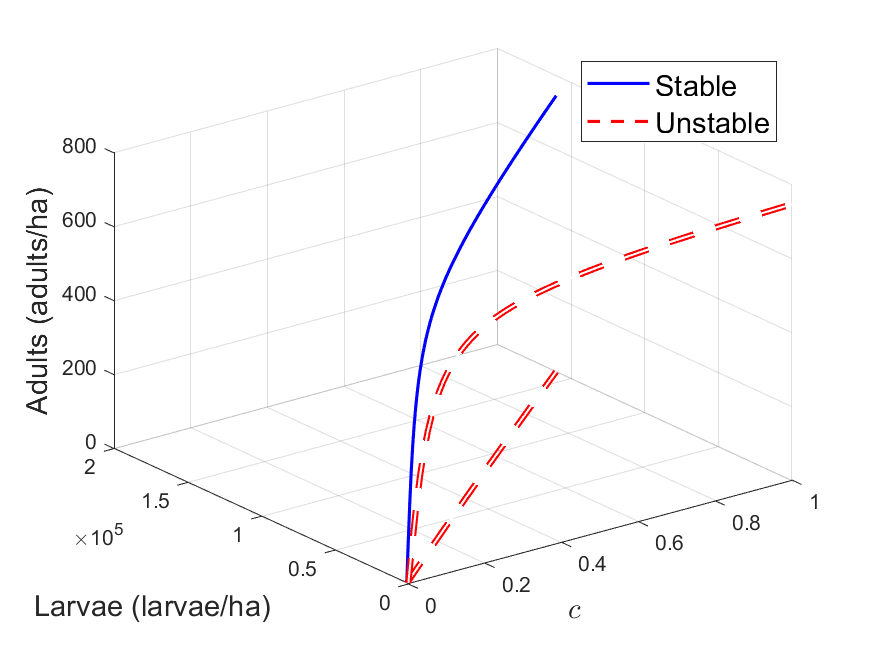}
    \caption{}
    \label{fig:4e}
\end{subfigure}%
\begin{subfigure}[t]{0.5\textwidth}
    \centering
    \includegraphics[width=\linewidth]{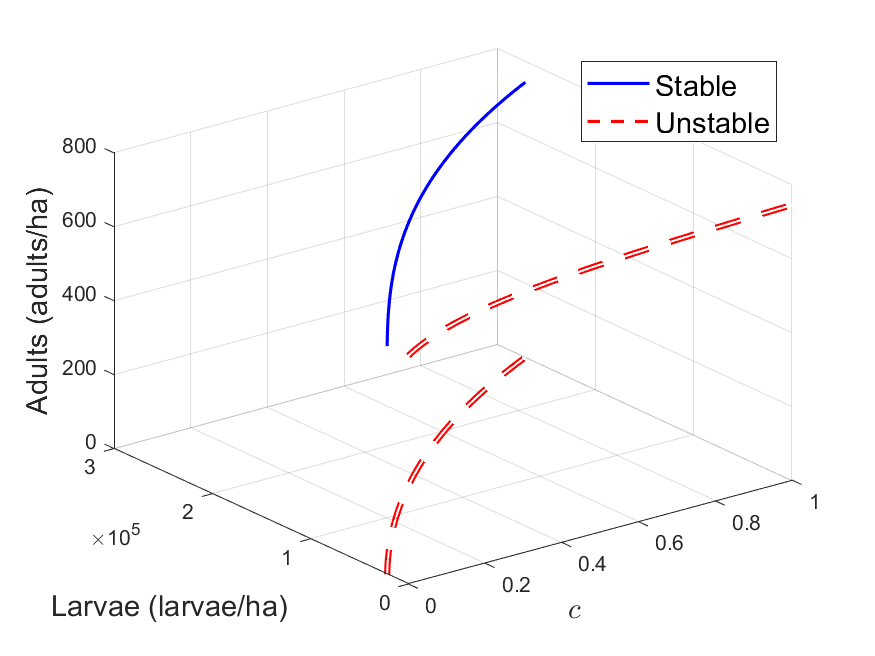}
    \caption{}
    \label{fig:4f}
\end{subfigure}%
\caption{Bifurcation diagrams for larval and adult populations with respect to $c$ using the northern parameter set (left column) and the southern parameter set (right column) from Table~\ref{paramvaluesnorthsouth} In each subfigure, the existence fixed point is shown as a single line, solid blue when stable and dashed red when unstable, while the 2-cycle is shown as a double line with the same color and line-style convention.}
\label{bifurcation_diagrams_c}
\end{figure}

\section{{Discussion and Conclusions}}
\label{discussion}

Ticks transmit numerous pathogens of public health concern. Identifying the ecological processes that regulate their populations is thus important for understanding patterns of disease risk. Tick population dynamics are governed by a sequence of events that depend not only on reproduction and environmental survival, but also on the ability of each life stage to locate and successfully feed on a host. Despite the importance of these host encounters, many mathematical models assume host-finding success is either constant or independent of tick density. In this study, we develop and analyze a stage-structured difference equation model that incorporates ratio-dependent host-finding success and regional differences in questing behavior. Parameterized for northeastern and southeastern ecosystems in the eastern United States, the model is used to investigate how host availability and questing behavior influence tick persistence, abundance, and the qualitative dynamics of tick populations.  

We first performed a qualitative analysis of two baseline models, one obtained when host-finding success is fixed and another obtained in the limiting case when the host-finding calibration parameter $z_j$ approaches infinity, yielding universal attachment success (Model~\eqref{simplifiedmodel}). Under both simplifications, tick persistence depended on reproduction and stage-specific survival. While period doubling in the form of a 2-cycle emerged, the 2-cycle was always unstable, indicating that in the absence of additional ecological mechanisms, the system tends toward a fixed equilibrium. These simplified models provide a baseline against which to measure the impact of ratio-dependent host-finding success and questing behavior introduced in the full model. 

Analysis of the full model (Model~\eqref{fullsystem}), which could only be completed numerically due to highly nonlinear dynamics, showed that introducing ratio-dependent host-finding success does not fundamentally change model dynamics qualitatively. The same threshold, $\mathcal{R}_d=1$, separates extinction from persistence, and the same $+1/-1$ bifurcation occurs at this threshold. Furthermore, the period-doubling solutions remain unstable, indicating that the introduction of ratio-dependent host-finding success alone is not sufficient to generate sustained oscillatory dynamics. This robustness of the threshold result shows that the demographic reproductive value remains the primary determinant of long-term dynamics, even when host and tick limitations are incorporated.

Although ratio-dependent attachment does not qualitatively change the dynamics of the system, it does alter quantitative results. In the simplified models, fixed-point population densities increase linearly with increasing values of $\mathcal{R}_d$. In contrast, the full model exhibits a saturation effect, where population growth slows as tick abundance increases relative to host availability. Consequently, tick and host abundances act as regulatory mechanisms that constrain population expansion even under favorable demographic conditions. These findings demonstrate that models assuming fixed or universal host-finding success may substantially overestimate tick densities. 

Numerical simulations showed that questing behavior strongly influences population abundance by determining which hosts are available for blood feeding. Increasing the questing parameter $c$, corresponding to greater time allocated to questing above the leaf litter, consistently increased population densities by improving host-finding success. Conversely, low values of $c$ reduced feeding opportunities, particularly in northern ecosystems.

The effects of questing behavior differed between northern and southern regions due to differences in host community composition. At low values of $c$, southern populations persisted at higher densities, whereas northern populations approached extinction. This reflects the greater abundance of below-litter hosts available to feeding ticks in the south. Northern communities do not support a sufficient number of these hosts, and thus tick populations that quest below the leaf litter are at risk of extinction. As $c$ increased, however, northern populations grew more rapidly and eventually exceeded southern populations, reflecting a greater abundance of above-litter hosts in the northern parameter set. 

Interestingly, because the assumed southern host density is higher overall compared to the northern one, southern tick populations exceeded northern populations over a large portion of the biologically relevant range of questing behavior ($0<c<0.83$). However, when comparing baseline $c$ values for each region ($c \approx 0.02$ in the south and $c \approx 0.18$ in the north), predicted southern population densities remain lower than northern densities, consistent with current field observations \cite{ginsberg2021lyme, foster2024density}. This finding suggests that differences in questing behavior as well as host community composition contribute to the observed geographic variation in tick abundance. Moreover, the model predicts that if southern populations exhibited questing behavior similar to the northern populations (without increased death from environmental factors), their equilibrium densities could exceed those observed in the north, which may have further implications for disease risk. %As noted by \cite{ginsberg2021lyme}, reported tick densities in the southern U.S. are often much lower than in the north; however, commonly used drag and flag sampling methods may underestimate southern populations because of regional differences in questing behavior. Our results highlight that these large observed differences in densities between regions may partially reflect sampling limitations rather than differences in population size. 

Overall, these findings highlight how two key mechanisms--ratio-dependent host attachment and questing behavior--contribute to tick population dynamics. First, ratio-dependent host-finding success primarily influences population size by limiting blood-meal availability as tick abundance increases relative to host density. Second, questing behavior determines both the frequency and composition of host encounters, influencing successful feeding and shaping regional differences in tick abundance.   

{Although pathogen transmission is not included in the present model, the demographic behaviors identified here provide an important foundation for understanding how to assess Lyme disease risk. Questing behavior can affect transmission through two pathways: by changing the total abundance of ticks that survive to later life stages and by changing the relative frequency with which immature tick feed on different host species. Because mammals and reptiles differ in their competence for \textit{B. burgdorferi} \citep{logiudice2003ecology, ginsberg2021lyme, giery2007role}, a shift in questing behavior may alter infection risk even when its effect on total tick abundance is low. Consequently, tick densities alone may be insufficient to characterize disease risk; the host composition of tick blood meals must also be considered.}  

This study {provides an initial framework for examining} how questing behavior and ratio-dependent attachment shape tick population dynamics; however, several limitations should be considered when interpreting the results. The host community considered here represents only a subset of the vertebrate species utilized by \textit{Ixodes scapularis} which are host generalists and feed on over 125 species \citep{keirans1996ixodes}. Birds were excluded because of limited density estimates and their relatively small contribution to overall tick demography \citep{giardina2000modeling, peterson2026host}, although they may nevertheless play an important role in tick dispersal \cite{hamer2010invasion, hamer2012wild, loss2016quantitative}. In addition, gaps in our understanding of southern tick populations and their hosts required approximating or borrowing several parameters from northern systems. Thus, although the model captures qualitative differences between northern and southern populations, uncertainty in the southern parameterization may affect the magnitude of the predicted responses. Finally, the model treats parameters independently when exploring their effects, despite the fact that many biological traits are likely correlated. For example, greater above-litter questing may increase host-finding opportunities while simultaneously decreasing environmental survival, particularly under the warmer and drier conditions characteristic of southern regions \citep{leal2020questing, eisen2016linkages}.

Several extensions could further develop this framework. Incorporating additional host species or logically linked parameter changes would provide a more comprehensive representation of tick ecology. {This work can also be extended to consider the alternative hypothesis that tick host-finding success is instead influenced by previous tick abundances rather than current abundance, reflecting host acquired tick resistance or learned behavior changes \citep{peterson2026modeling}}. Future work could incorporate pathogen transmission dynamics, which would allow the model to connect the effects of questing behavior and ratio-dependent attachment to patterns of disease transmission under northern and southern conditions.

Ultimately, this study shows that {population persistence depends on demographic reproduction through environmental survival, questing behavior, and host availability}. Questing behavior regulates opportunities for host encounters and determines which host communities support tick populations, whereas ratio-dependent attachment constrains population growth through competition for finite host resources. Incorporating both processes into mathematical models provides a more biologically realistic framework for predicting tick abundance. 

\appendix

\section{Sequence of Seasonal Events}\label{ap:timeline}
The annual update is implemented through the following sequence of events, with each event occurring at the indicated fractional time step.
\begin{enumerate}[1)]
    \item Nymphs experience environmental mortality and host-finding failure:
    \begin{align*}
        N\left(t+\frac{1}{7}\right)&=s_NN\left(t\right)\left(1-\exp\left(-\frac{z_N}{L(t)}\sum_{i=1}^6 b_{i,N}\sigma_{i,N}H_i(t))\right)\right)
    \end{align*}
    \item Surviving nymphs attach to and feed on hosts:
    \begin{align*}
        N\left(t+\frac{2}{7}\right) = N\left(t+\frac{1}{7}\right)
    \end{align*}
    \item The first half of annual host recruitment occurs for groups 1--4; larvae emerge, and nymphs transition to the adult stage:
    \begin{align*}
        H_{i}\left(t+\frac{3}{7}\right)&=H_{i}\left(t+\frac{2}{7}\right)+\frac{\Lambda_i}{2} \quad \text{for $i=1,2,3,4$}\\
         H_{5}\left(t+\frac{3}{7}\right)&=H_5\left(t+\frac{2}{7}\right)+\Lambda_5\\
        H_6\left(t+\frac{3}{7}\right)&=H_6\left(t+\frac{2}{7}\right)+\Lambda_6\\~\\
           L\left(t+\frac{3}{7}\right) &= \frac{r\frac{A\left(t-\frac{1}{7}\right)}{2}}{\frac{\frac{A\left(t-\frac{1}{7}\right)}{2}}{H_6\left(t-\frac{1}{7}\right)}+K}\\~\\
           A\left(t+\frac{3}{7}\right)&=N\left(t+\frac{2}{7}\right)\\~\\
            N\left(t+\frac{3}{7}\right)&=0\\
    \end{align*}
    \item Larvae experience environmental mortality and host-finding failure, while the first half of annual host mortality occurs:
    \begin{align*}
         L\left(t+\frac{4}{7}\right)&=s_LL\left(t+\frac{3}{7}\right)\left(1-\exp\left(-\frac{z_L}{B\left(t+\frac{3}{7}\right)\sum_{i=1}^6 b_{i,L}H_i\left(t+\frac{3}{7}\right)}\right)\right)\\~\\
        H_{i}\left(t+\frac{4}{7}\right) &= H_{i}\left(t+\frac{3}{7}\right)\exp\left(-\frac{\mu_{i}}{2}\right) \quad \text{for $i=1,2,3,4,5$}\\
        H_6\left(t+\frac{4}{7}\right)&= H_6\left(t+\frac{3}{7}\right)\exp\left(-\frac{\mu_{6}}{2}\right)
    \end{align*}    
    \item Surviving larvae attach to and feed on hosts while adults experience environmental mortality and host-finding failure:
    \begin{align*}
        L\left(t+\frac{5}{7}\right)&=L\left(t+\frac{4}{7}\right)\\~\\
        A\left(t+\frac{5}{7}\right)&=s_AA\left(t+\frac{4}{7}\right)\left(1-\exp\left(-\frac{z_A}{N\left(t+\frac{2}{7}\right)} b_{6,A}\sigma_{6,A}H_6\left(t+\frac{4}{7}\right)\right)\right)
    \end{align*}
    \item Larvae transition to nymphs, adults reproduce, and the remaining annual recruitment occurs for host groups 1--4:
    \begin{align*}
        N\left(t+\frac{6}{7}\right)&=L\left(t+\frac{5}{7}\right)\\~\\
        L\left(t+\frac{6}{7}\right)&=0\\~\\
        H_{i}\left(t+\frac{6}{7}\right)&=H_{i}\left(t+\frac{5}{7}\right)+\frac{\Lambda_i}{2} \quad \text{for $i=1,2,3,4$}\\~\\
        A\left(t+\frac{6}{7}\right)&=A\left(t+\frac{5}{7}\right)
    \end{align*}
    \item Adult ticks are removed following reproduction, and the remaining annual host mortality is applied:
    \begin{align*}
        A\left(t+1\right)&=0\\~\\
        H_{i}\left(t+1\right) &= H_{i}\left(t+\frac{6}{7}\right)\exp\left(-\frac{\mu_{i}}{2}\right) \quad \text{for $i=1,2,3,4,5$}\\
        H_6\left(t+1\right)&= H_6\left(t+\frac{6}{7}\right)\exp\left(-\frac{\mu_{6}}{2}\right)
    \end{align*}
\end{enumerate}

\section{Qualitative Analysis Derivation}
\subsection{{Baseline Model Stability Analysis}}\label{ap:qual_analysis}
Assuming tick host-finding success is constant (either fixed at a baseline value or with full host-finding success), the baseline tick subsystem of Model~\eqref{fullsystem} is
\begin{equation*}
    \begin{aligned}
    L(t+1)&=\frac{rA(t)/2}{\frac{A(t)/2}{H_{6_{\infty}}}+K}s_L\\
    N(t+1)&=s_NL(t)\\
    A(t+1)&=s_As_NL(t).
    \end{aligned}
\end{equation*}
Noting that the $LA$ subsystem decouples from the $LNA$ system, we can perform the analysis on just the two-dimensional system. Solving for equilibria gives:
\begin{align*}
    (L_0,A_0) = (0,0), \quad (L_\infty,A_\infty) = \left(\frac{2H_{6_{\infty}}(rS-K)}{s_Ns_A}, 2H_{6_{\infty}}(rS-K)\right),
\end{align*}
where $S=s_Ls_Ns_A/2$. The second fixed point exists biologically iff $\frac{rS}{K}>1$.

The Jacobian of the system is
\begin{align*}
    J(L,A)=
    \begin{bmatrix}
        0&\frac{2H_{6_\infty}^2Krs_L}{(A+2H_{6_\infty}K)^2}\\
        s_Ns_A&0
    \end{bmatrix}.
\end{align*}
At the extinction fixed point $(0,0)$:
\begin{align*}
    J(L_0,A_0)=
    \begin{bmatrix}
        0&\frac{rs_L}{2K}\\
        s_Ns_A&0
    \end{bmatrix},
\end{align*}
with eigenvalues 
\begin{align*}
    \lambda_{1,2}=\pm \sqrt{\frac{rS}{K}}.
\end{align*}
Thus, the extinction fixed point is LAS iff $\frac{rS}{K}<1$.

At the existence fixed point:
\begin{align*}
    J(L_\infty,A_\infty)=
    \begin{bmatrix}
        0&\frac{s_LK}{2rS^2}\\
        s_Ns_A&0
    \end{bmatrix},
\end{align*}
yielding eigenvalues
\begin{align*}
     \lambda_{1,2}=\pm \sqrt{\frac{K}{rS}},
\end{align*}
so it is LAS iff $rS/K>1$. 

Note that in both the evaluation of the extinction and existence fixed point, the eigenvalues are $-1$ and $1$ when $rS/K=1$. This confirms analytically the $+1$/$-1$ bifurcation that occurs at $rS/K=1$, which is described in~\ref{ap:full_qual_analysis}. 

\subsection{Second Generation System Analysis - Baseline Model}\label{ap:second_gen_analysis}
To assess the potential for period-doubling behavior, we derive the second-generation system of the baseline model:
\begin{equation*}
    \begin{aligned}
    L(t+2)&=\frac{r\frac{s_Ns_AL(t)}{2}}{\frac{\frac{s_Ns_AL(t)}{2}}{H_{6_{\infty}}}+K}s_L,\\
    A(t+2)&=s_As_N\frac{rA(t)/2}{\frac{A(t)/2}{H_{6_{\infty}}}+K}s_L.
    \end{aligned}
\end{equation*}
This system admits two additional fixed points beyond those of the first-generation system:
\begin{align*}
    &\left(L_{1_\infty},A_{1_\infty}\right)=\left(\frac{2H_{6_{\infty}}(rS-K)}{s_Ns_A},0\right),\quad \left(L_{2_\infty},A_{2_\infty}\right)=\left(0,2H_{6_{\infty}}(rS-K)\right).
\end{align*}
These fixed points correspond to the two branches of a two-cycle and exist for $rS/K>1$.

To assess stability, we compute the Jacobian matrix of the second-generation map:
\begin{align*}
    J(L,A)=
    \begin{bmatrix}
        \frac{2H_{6_\infty}^2Krs_Ls_Ns_A}{(s_Ns_AL+2H_{6_\infty}K)^2}&0\\[5pt]
        0&\frac{2H_{6_\infty}^2Krs_Ls_Ns_A}{(A+2H_{6_\infty}K)^2}
    \end{bmatrix}.
\end{align*}
Evaluating the Jacobian at $\left(L_{1_\infty},A_{1_\infty}\right)$ yields
\begin{align*}
    J(L_{1_\infty},A_{1_\infty})=
    \begin{bmatrix}
        \frac{K}{rS}&0\\[5pt]
        0&\frac{rS}{K}
    \end{bmatrix}.
\end{align*}
Because this matrix is diagonal, its eigenvalues are given by the diagonal entries:  
\begin{align*}
    \lambda_1 = \frac{K}{rS},\quad \lambda_2 = \frac{rS}{K}.
\end{align*}
The two eigenvalues cannot both have magnitude less than one for any set of parameter values. Consequently, this branch of the two-cycle cannot be LAS. Hence, the two-cycle is never stable.

\subsection{{Full Model Fixed Point Analysis}}\label{ap:full_qual_analysis}

Consider the system
\begin{equation*}
    \begin{aligned}
    L(t+1)&=\frac{rA(t)/2}{\frac{A(t)/2}{H_6(t)}+K}s_LS_L(t),\\
    N(t+1)&=s_NS_N(t)L(t),\\
    A(t+1)&=s_AS_A(t)s_NS_N(t)L(t),
    \end{aligned}
\end{equation*}
where 
\begin{align*}
    S_L(t)&=\begin{cases}
    1 - e^{- \frac{z_L}{\frac{rA(t)/2}{\frac{A(t)/2}{H_6(t)}+K}}\Sigma_L}, & A(t)>0\\
    1, & A(t)=0\end{cases},\\
    S_N(t)&=\begin{cases}
    1 - e^{- \frac{z_N}{L(t)}\Sigma_N}, & L(t)>0\\
    1, & L(t)=0\end{cases},\\
    S_A(t)&=\begin{cases}
    1 - e^{- \frac{z_A}{s_NS_N(t)L(t)}\Sigma_A}, & L(t)>0\\
    1, & L(t)=0\end{cases}.\\
\end{align*}

Similar to the baseline model, the $LA$ subsystem decouples from the full model and can be analyzed separately. 

The Jacobian matrix of this system can be written as
\begin{equation}\label{Jacobian}
\begin{aligned}J=
    \begin{bmatrix}
        0 & \frac{L_{\infty}}{A_\infty}\left(\frac{K}{\frac{A_\infty/2}{H_{6_\infty}}+K}\right)\left(1+f(x_L)\right)\\
        \frac{A_\infty}{L_\infty}\left(1+f(x_N)\right)\left(1+f(x_A)\right) & 0
    \end{bmatrix},
\end{aligned}
\end{equation}
where $f(x)=\frac{-xe^{-x}}{1-e^{-x}}$ and 
\begin{align*}
    x_L=\frac{z_L \Sigma_L}{B_\infty}, \quad
    x_N = \frac{z_N\Sigma_N}{L_\infty}, \quad
    x_A = \frac{z_A\Sigma_A}{s_NS_{N_\infty}L_\infty}.
\end{align*}

Evaluating the Jacobian at the extinction fixed point by taking the entrywise limit as $(L_\infty,A_\infty)\to(0,0)$, we get
    \begin{align*}
    J(0,0)=
    \begin{bmatrix}
        0&\frac{rs_L}{2K}\\
        s_Ns_A&0
    \end{bmatrix}.
\end{align*}
The eigenvalues of this matrix are
\begin{align*}
\lambda_{1,2}=\pm \sqrt{\frac{rs_Ls_Ns_A}{2K}}.
\end{align*}
Thus, the extinction fixed point is LAS when $rS/K<1$ with $S = s_Ls_Ns_A/2$.

Note that the eigenvalues here are $-1$ and $1$ when $rS=K$, which confirms analytically the $+1$/$-1$ bifurcation discussed in Section~\ref{populationdynamics_full}.

Although the equations in Model~\eqref{fullsystem} are transcendental and thus cannot be solved explicitly for closed-form solutions, we can still assess stability conditions via the Jacobian~\eqref{Jacobian} and the Jury criterion. For a two-dimensional system, the Jury criterion guarantees stability when 
\begin{align*}
    |\tr(J)|<\det(J)+1<2.
\end{align*}
For our system, this corresponds to
\begin{align*}
-1<-\left(\frac{K}{\frac{A_\infty/2}{H_{6_\infty}}+K}\right)(1+f(x_L))(1+f(x_N))(1+f(x_A))<1.
\end{align*}
Because $f(x)$ is between $-1$ and $0$, we have that $(1+f(x))$ is between $0$ and $1$. Since the fractional term $\frac{K}{\frac{A_\infty/2}{H_{6_\infty}}+K}$ is also between $0$ and $1$, the compound inequality holds, and thus the positive existence fixed point is LAS when it exists, which corresponds to when $\frac{rS}{K}>1$ as detailed in the next section.

\subsection{{Existence of Positive Fixed Point in Full Model}}\label{ap:existence}
We now state and prove a necessary and sufficient condition for the existence of a unique positive fixed point of Model~\eqref{fullsystem}. 

\begin{thm}
Model~\eqref{fullsystem} has a unique positive fixed point if and only if
\[
\mathcal{R}_d=\frac{rS}{K} > 1,
\]
where $S=s_L s_N s_A/2$.
\end{thm}

\begin{proof}
Recall the full system:
\begin{align*}
    L(t+1)&=B(A)s_L\left(1 - e^{- \frac{z_L}{B(A)}\Sigma_L}\right),\\
    N(t+1)&=s_N\left(1 - e^{- \frac{z_N}{L(t)}\Sigma_N}\right)L(t),\\
    A(t+1)&=s_A\left(1 - e^{- \frac{z_A}{N(t+1)}\Sigma_A}\right)N(t+1),
\end{align*}
for $L,N,A>0$ where $B(A)=\frac{rA(t)/2}{\frac{A(t)/2}{H_6(t)}+K}$.
A positive fixed point $(L_\infty,N_\infty,A_\infty)$ must satisfy
\begin{align*}
    L_\infty&=B(A_\infty)s_L\left(1 - e^{- \frac{z_L}{B(A_\infty)}\Sigma_L}\right),\\
    N_\infty&=s_N\left(1 - e^{- \frac{z_N}{L_\infty}\Sigma_N}\right)L_\infty,\\
    A_\infty&=s_A\left(1 - e^{- \frac{z_A}{N_\infty}\Sigma_A}\right)N_\infty.
\end{align*}
Define
\begin{align*}
    f_1(x) &= B(x)s_L\left(1 - e^{- \frac{z_L}{B(x)}\Sigma_L}\right)\\
    f_2(x) &= s_N\left(1 - e^{- \frac{z_N}{x}\Sigma_N}\right)x\\
    f_3(x) &= s_A\left(1 - e^{- \frac{z_A}{x}\Sigma_A}\right)x.
\end{align*}
Then existence of a positive fixed point of Model~\eqref{fullsystem} is equivalent to existence of a positive solution of $A = f_3(f_2(f_1(A)))$. Define $G(A)=f_3(f_2(f_1(A)))$.

Note that each function $f_i$ for $i\in \{1,2,3\}$ satisfies
\begin{itemize}
    \item $f_i(x)\to 0$ as $x \to 0$,
    \item $f_i'(x)>0$ for $x>0$,
    \item $f_i''(x)<0$ for $x>0$.
\end{itemize}

Thus, each $f_i$ is increasing and concave on $(0,\infty)$, and thus so is their composition $G$. In addition, $G(x) \to 0$ as $x \to 0$.

Because $G$ is increasing, concave, and zero at the origin, the equation $A=G(A)$ has at most one positive solution. 

To determine whether such a solution exists, we examine the slope at the origin. In particular, a positive fixed point exists iff $G'(0)>1$.

By the chain rule, 
\begin{align*}
    G'(0)&=f_3'(0) f_2'(0) f_1'(0)\\
    &=(s_A)
\left(s_N\right)
\left(\frac{rs_L}{2K}\right)\\
&=\frac{r s_L s_N s_A}{2K}\\
&=\frac{rS}{K},
\end{align*}
where $S=s_Ls_Ns_A/2$.

Therefore, $G'(0)>1$ if and only if $rS/K>1$. 

Thus, Model~\eqref{fullsystem} has a unique positive fixed point if and only if $rS/K>1$.
\end{proof}

\subsection{{Second Generation System - Full Model}}\label{ap:second_gen_full}
To assess the stability of a 2-cycle of Model~\eqref{fullsystem}, we consider the second-generation map iterating each equation once. This yields
\begin{align*}
    L(t+2)&=B(t+1)s_L\left(1-e^{-\frac{z_L}{B(t+1)}\Sigma_L}\right)\\
    A(t+2)&=B(t)s_Ls_Ns_A\left(1-e^{-\frac{z_L}{B(t)}\Sigma_L}\right)\left(1-e^{-\frac{z_N}{B(t)s_LS_L(t)}\Sigma_N}\right)\left(1-e^{-\frac{z_A}{s_NS_N(t+1)B(t)s_LS_L(t)}\Sigma_A}\right),
\end{align*}
where
\begin{align*}
    B(t+1) &= \frac{rs_Ns_AS_N(t)S_A(t)L(t)/2}{\frac{s_Ns_AS_N(t)S_A(t)L(t)/2}{H_{6_\infty}}+K}
\end{align*}
and
\begin{align*}
    S_N(t+1) &= \left(1-e^{-\frac{z_N}{B(t)s_LS_L(t)}}\right).
\end{align*}
The Jacobian of the second-generation system evaluated at the 2-cycle is

\begin{align*}
    J = \begin{bmatrix}
        \frac{r s_L s_N s_A S_{L_{2_\infty}} S_{N_\infty} S_{A_\infty} /2}{\frac{s_N s_A S_{N_\infty} S_{A_\infty} L_\infty/2}{H_{6_\infty}}+K} \cdot J_{LL} & 0\\
        0 & \frac{r s_L s_N s_A S_{L_\infty} S_{N_{2_\infty}} S_{A_{2_\infty}} /2}{\frac{A_{\infty}/2}{H_{6_\infty}}+K} \cdot J_{AA}
    \end{bmatrix},
\end{align*}
where
\begin{align*}
S_{L_{2_\infty}} &= \left(1-e^{-\frac{z_L}{B_{2_\infty}}\Sigma_L}\right),\\
S_{N_{2_\infty}} &= \left(1-e^{-\frac{z_N}{B_\infty s_LS_{L_\infty}}\Sigma_N}\right),\\
S_{A_{2_\infty}} &= \left(1-e^{-\frac{z_A}{s_NS_{N_{2_\infty}}B_\infty s_LS_{L_\infty}}\Sigma_A}\right),
\end{align*}
and
\begin{align*}
    J_{LL}&= \frac{K}{\frac{s_Ns_AS_{N_\infty}S_{A_\infty}L_\infty/2}{H_{6_\infty}}+K}(1+f(x_{L_2}))(1+f(x_N))(1+f(x_A)),\\
    J_{AA} &= \frac{K}{\frac{A_\infty/2}{H_{6_\infty}}+K}(1+f(x_{L}))(1+f(x_{N_2}))(1+f(x_{A_2})).
\end{align*}
Here, $f(x)$, $x_L$, $x_N$, and $x_A$ are as defined in~\ref{ap:full_qual_analysis} and
\begin{align*}
    x_{L_2}=\frac{z_L  \Sigma_L}{B_{2_\infty}}, \quad
    x_{N_2} = \frac{z_N \Sigma_N}{B_\infty s_L S_{L_\infty}}, \quad
    x_{A_2} = \frac{z_A \Sigma_A}{B_\infty s_L s_N s_A S_{L_\infty}S_{N_{2_\infty}}},
\end{align*}
with 
\begin{align*}
    B_\infty = \frac{rA_\infty/2}{\frac{A_\infty/2}{H_{6_\infty}}+K}, \quad 
    B_{2_\infty} = \frac{rs_Ns_AS_{N_\infty}S_{A_\infty}L_\infty/2}{\frac{s_Ns_AS_{N_\infty}S_{A_\infty}L_\infty/2}{H_{6_\infty}}+K}.
\end{align*}
Evaluating the Jacobian at the two ends of the 2-cycle
\begin{align*}
    (L_1,A_1)=(L_\infty,0), \quad (L_2,A_2)= (0,A_\infty),
\end{align*}
yields
\begin{align*}
    J(L_1,A_1)=\begin{bmatrix}
        J_{LL} & 0\\
        0 & \frac{rS}{K}
    \end{bmatrix}, \quad J(L_2,A_2) = \begin{bmatrix}
        \frac{rS}{K} & 0\\
        0 & J_{AA}
    \end{bmatrix},
\end{align*}
since
\begin{align*}
    \frac{r s_L s_N s_A S_{L_{2_\infty}} S_{N_\infty} S_{A_\infty} /2}{\frac{s_N s_A S_{N_\infty} S_{A_\infty} L_\infty/2}{H_{6_\infty}}+K} = \frac{L_\infty}{L_\infty}=1, \quad \text{and} \quad \frac{r s_L s_N s_A S_{L_\infty} S_{N_{2_\infty}} S_{A_{2_\infty}} /2}{\frac{A_\infty/2}{H_{6_\infty}}+K} = \frac{A_\infty}{A_{\infty}} = 1,
\end{align*}
and $S_{L_\infty}$, $S_{N_{2_\infty}}$, $S_{A_{2_\infty}}$ approach one as $A_\infty$ approaches zero while $S_{L_{2_\infty}}$, $S_{N_\infty}$, $S_{A_\infty}$ approach one as $L_\infty$ approaches zero.

The eigenvalues of these matrices are the diagonal entries, and because both $J_{LL}$ and $J_{AA}$ are positive and less than one, the condition for stability is $\frac{rS}{K}<1$. However, the 2-cycle only exists for $\frac{rS}{K}>1$. Thus, the 2-cycle is always unstable.

\section{Calculation of Derived Parameter Values}\label{ap:paramcalc}

Not all parameters in Table~\ref{paramvaluesnorthsouth} were available directly from published studies. Several were instead inferred from the model structure, baseline host communities, and modeling assumptions. The procedures used to obtain these derived parameter values are summarized below.

Baseline host densities, together with host recruitment $\Lambda_i$, were used to determine the mortality parameters through the fixed point conditions of the host equations derived in Section~\ref{modelanalysis}. Specifically, the fixed point host densities satisfy
\begin{equation*}
    \begin{aligned}
        H_{i_\infty} &= \frac{\Lambda_i}{2} \left(\frac{e^{-\mu_i}+e^{-\frac{\mu_i}{2}}}{1-e^{-\mu_i}}\right), \quad \text{for $i \in \{1,2,3,4\}$},\\
        H_{5_{\infty}}&=\frac{\Lambda_{5}e^{-\mu_5}}{1-e^{-\mu_5}},\\
        H_{6_{\infty}}&=\frac{\Lambda_{6}e^{\frac{-\mu_6}{2}}}{1-e^{-\mu_6}}.
    \end{aligned}
\end{equation*}
Solving these equations for $\mu_i$ gives
\begin{align*}
    \mu_i &= 2\ln\left(\frac{\Lambda_i}{2H_{i_\infty}}+1\right), \quad \text{for $i\in \{1,2,3,4\}$},\\
    \mu_5 &= \ln\left(\frac{\Lambda_5}{H_{5_\infty}+1}\right),\\
    \mu_6 &= 2\ln\left(\frac{\Lambda_6 +\sqrt{\Lambda_6^2+4H_{6_\infty}^2}}{2H}\right).
\end{align*}

Because each host category contains several species, literature estimates first had to be expressed in terms of a single representative host before they could be incorporated into the model. Following the framework introduced in Section~\ref{development}, we used host-burden data compiled by \cite{ginsberg2021lyme}. Northern preference estimates were based on observations from Massachusetts, New Jersey, and Wisconsin, whereas southern estimates were obtained from North Carolina, South Carolina, and Tennessee, which together provided the most complete regional dataset.

Relative host preference was estimated separately for each tick life stage by dividing the total number of ticks recovered from a host species by the total number of individuals of that species examined during the corresponding questing season, including hosts without attached ticks. These per-host burdens were then adjusted using the assumed regional questing behavior. To represent cumulative exposure across an average ten-week questing season, each coefficient was multiplied by ten. Because larvae, nymphs, and adults remain attached for different lengths of time, additional stage-specific scaling factors of 2, 1.5, and 1 were applied based on typical attachment durations of 3--5, 4--6, and 6--8 days, respectively \citep{narasimhan2024laboratory, nuss2017rearing}. This procedure yielded the baseline preferences $\hat{b}_{i,j}$.

These preference estimates were subsequently used to express every species within a host group in terms of the selected representative species. Conversion factors were computed by comparing the observed burden of each species with that of the representative host, and reported densities were multiplied by these factors so that all host abundances were expressed as a common host.

Since host preference varies across tick life stages, the representative host densities obtained above differ for larvae, nymphs, and adults. To retain a single baseline host density in the model, we introduced scaling parameters $\sigma_{i,j}$ that relate the larval-equivalent density to the corresponding nymphal and adult values. Each scaling constant was computed as the ratio of the stage-specific density to the larval density for the same host group, yielding the values reported in Table~\ref{paramvaluesnorthsouth}.

The calibration parameters $z_j$ were chosen so that the modeled probability of successfully locating a host reproduced baseline estimates of stage-to-stage survival. Their calculation combines the baseline tick densities, host abundances, preference coefficients, and scaling parameters listed in Table~\ref{paramvaluesnorthsouth}. An additional was made because the available field estimates correspond to questing ticks that have already survived off-host mortality. In contrast, the host-finding function $S_j(t)$ is applied before environmental survival in the model. Consequently, the baseline questing densities were converted to pre-survival stage abundances by dividing each observed density by its corresponding environmental survival probability. For example, the larval abundance entering the host-finding calculation is $L_{\text{base}}/s_L$ rather than the observed questing density $L_{\text{base}}$.

Published estimates indicate that approximately 10\% of larvae, 10\% of nymphs, and 30\% of adults successfully survive and transition to the next life stage \citep{daniels2000estimating, ostfeld2023mouse}. Combining these survival estimates with the environmental survival probabilities $s_j$ yields the following expressions for the calibration constants:
\begin{align*}
    z_L &= -\frac{L_{\text{base}}}{s_L}\frac{\ln(1-\frac{0.1}{s_L})}{\Sigma_{L_{\text{base}}}}\\
    z_N &= -\frac{N_{\text{base}}}{s_N}\frac{\ln(1-\frac{0.1}{s_N})}{\Sigma_{N_{\text{base}}}}\\
    z_A &= -s_NS_{N_{\text{base}}}\frac{N_{\text{base}}}{s_N}\frac{\ln(1-\frac{0.3}{s_A})}{\Sigma_{A_{\text{base}}}},
\end{align*}
where $\Sigma_{j_{\text{base}}}$ denotes the baseline host-availability term defined in Section~\ref{model}.

\section{Southern Host Densities}\label{ap:southernhosts}
\begin{table}[!h]
    \centering
    \resizebox{\textwidth}{!}{
    \begin{tabular}{c|c|c|c|c}
    \toprule
        \textbf{Species} & \makecell{\textbf{Population}\\ \textbf{Density}\\\textbf{(no./ha)}} & \makecell{\textbf{Location}} & \makecell{\textbf{Habitat Type}} & \makecell{\textbf{Time of Year}}\\
        \midrule
        \makecell{White-footed\\mouse} &\makecell{10.8--46.1 \citep{buckner1985response}\\36.5 \citep{scarlett2004acorn}\\19.5 \citep{scarlett2004acorn}\\11.7 \citep{scarlett2004acorn}\\2.72 \citep{howell1954populations}}&\makecell{North Carolina\\North Carolina\\North Carolina\\
        North Carolina\\Tennessee}& \makecell{Forest\\Forest\\Forest\\Forest\\Grassland}&\makecell{June--August\\February--March\\December\\September--October\\July--August}\\
        \hline
        \makecell{Cotton mouse} & \makecell{6.10  \citep{mccarley1954fluctuations}\\3.30 \citep{mccarley1954fluctuations}\\2.22 \citep{mccarley1954fluctuations}\\1.84 \citep{mabry2003influence}\\1.27 \citep{mabry2003influence}\\1.24 \citep{mccarley1954fluctuations}\\0.52 \citep{mabry2003influence}} & \makecell{Texas\\Texas\\Texas\\South Carolina\\South Carolina\\Texas\\South Carolina} & \makecell{Forest\\Forest\\Forest\\Forest\\Forest\\Forest\\Forest} & \makecell{January--March\\April--June\\July--September\\March--May\\June--August\\October--December\\September--December}\\
        \hline
        \makecell{Eastern\\chipmunk} & \makecell{4.3--78.6 \citep{wolff1996population}} & \makecell{Virginia} & \makecell{Forest} & \makecell{April--October}\\
        \hline
        \textit{Sorex} shrew & \makecell{30--44 \citep{french1980natural}} & \makecell{Alabama} & \makecell{Forest} & \makecell{October--June}\\
        \hline
        \makecell{Short-tailed\\shrew} &\makecell{30.6 \citep{kitchings1981habitat}\\18.4 \citep{kitchings1981habitat}}&\makecell{Tennessee\\Tennessee}& \makecell{Forest\\Forest}&\makecell{June--August\\October--November}\\
        \hline
        \makecell{Red-backed\\vole} &\makecell{50 \citep{adams1991changes}}&\makecell{North Carolina}& \makecell{Forest}&\makecell{August}\\
        \hline
        Gray squirrel & \makecell{8.1 \citep{williams2011comparison}\\6.5 \citep{williams2011comparison}\\5.2 \citep{williams2011comparison}\\4.9 \citep{williams2011comparison}\\0.59--3.3 \citep{barkalow1970vital}\\0.64--2.5 \citep{barkalow1970vital}} & \makecell{Georgia\\Georgia\\Georgia\\Georgia\\North Carolina\\North Carolina} & \makecell{Forest\\Urban park\\Residential woodlots\\Forest\\Woodland\\Woodland} & \makecell{October--March\\October--March\\October--March\\October--March\\Fall\\Spring}\\
        \hline
        Red squirrel & \makecell{0.1--1.5 \citep{stevens1999evaluation}} & \makecell{Tennessee} & \makecell{Forest} & \makecell{Annual}\\
        \bottomrule
        \end{tabular}
        }
        \caption{Host densities in units of individuals/ha for the southeastern U.S. (continued on the following page).}
        \label{southernhostdensities}
        \end{table}
        %\newpage
        \begin{table*}[ht!] \centering \resizebox{\textwidth}{!}{
    \begin{tabular}{c|c|c|c|c}
        \toprule
        \textbf{Species} & \makecell{\textbf{Population}\\ \textbf{Density}\\\textbf{(no./ha)}} & \makecell{\textbf{Location}} & \makecell{\textbf{Habitat Type}} & \makecell{\textbf{Time of Year}}\\
        \midrule
        \makecell{Five-lined\\skink} & \makecell{121--227 \citep{fitch1954life}} & \makecell{Kansas} & \makecell{Woodland} & \makecell{Annual}\\ \hline
        \makecell{Eastern\\fence lizard} & \makecell{47--62.5 \citep{parker1994demography}\\29.5--40 \citep{parker1994demography}\\7.66 \citep{tinkle1972sceloporus}\\1.65 \citep{tinkle1972sceloporus}\\\vspace{0.1cm}} & \makecell{Mississippi\\Mississippi\\South Carolina\\Texas\\\vspace{0.1cm}} & \makecell{Forest\\Forest\\Cleared area\\Rolling terrain\\with scattered rock} & \makecell{April--June\\July--September\\Annual\\Annual\\\vspace{0.1cm}}\\
        \hline
        \makecell{Ground skink} & \makecell{329.9--593.1 \citep{brooks1967population}\\504.1--583.2 \citep{brooks1967population}\\476.9--494.2 \citep{brooks1967population}\\388.0--449.7 \citep{brooks1967population}\\66.1 \citep{akin1998fourier}} & \makecell{Florida\\Florida\\Florida\\Florida\\Louisiana} & \makecell{Forest\\Forest\\Forest\\Forest\\Forest} & \makecell{July--September\\October--December\\January--March\\April--June\\July--August}\\
        \hline
        %\makecell{Southeastern\\five-lined skink} & \makecell{} & \makecell{} & \makecell{} & \makecell{}\\
        %\hline
        %\makecell{Striped skunk} & \makecell{0.05 \citep{stout1974striped}\\0.05 \citep{stout1974striped}\\0.04 \citep{stout1974striped}\\0.04 \citep{stout1974striped}\\0.002--0.007 \citep{hansen2004population}} & \makecell{Virginia\\Virginia\\Virginia\\Virginia\\Texas} & \makecell{Field/Forest\\Field/Forest\\Field/Forest\\Field/Forest\\Cropland} & \makecell{January--March\\October--December\\April--June\\July--September\\May--June}\\
        %\hline
        \makecell{White-tailed\\deer} & \makecell{0.55 \citep{wathen1989white}\\0.37--0.48 \citep{wathen1989white}\\0.37--0.40 \citep{wathen1989white}\\0.20 \citep{wathen1989white}} & \makecell{Tennessee\\Tennessee\\Tennessee\\Tennessee} & \makecell{Field/Woodlot\\Field/Woodlot\\Field/Woodlot\\Field/Woodlot} & \makecell{March--May\\September--November\\June--August\\December--February}\\
        \hline
        \makecell{Virginia\\opossum} & \makecell{0.11--0.29 \citep{wolcott2011population}\\0.10 \citep{conner1983scent}\\0.09 \citep{stout1974ecology}\\0.07 \citep{stout1974ecology}\\0.06 \citep{stout1974ecology}\\0.04 \citep{stout1974ecology}\\0.03 \citep{bernasconi2022influence}\\0.03 \citep{bernasconi2022influence}\\0.01 \citep{bernasconi2022influence}\\0.01 \citep{bernasconi2022influence}} & \makecell{Tennessee\\Florida\\Virginia\\Virginia\\Virginia\\Virginia\\South Carolina\\South Carolina\\South Carolina\\South Carolina} & \makecell{Forest\\Woodland\\Field/Woodland\\Field/Woodland\\Field/Woodland\\Field/Woodland\\Swamp\\Wetland\\Forest\\Forest} & \makecell{January--March\\December--March\\July--September\\April--June\\October-December\\January--March\\January--May\\January--May\\January--May\\January--May}\\
        \hline
        \makecell{Raccoon} & \makecell{0.08--1.05\\ \citep{sonenshine1972contrasts}\\0.12 \citep{johnson1969biology}\\0.05--0.12 \citep{keeler1978some}\\0.1--0.11 \citep{conner1983scent}\\0.07 \citep{nottingham1985raccoon}\\0.05 \citep{rabinowitz1981ecology}\\0.05 \citep{nottingham1985raccoon}\\0.01--0.04 \citep{keeler1978some}} & \makecell{Virginia\\\\Alabama\\Tennessee\\Florida\\Tennessee\\Tennessee\\Tennessee\\Tennessee} & \makecell{Woodland\\\\Woodland\\Forest\\Woodland\\Woodland\\Woodland\\Woodland\\Forest} & \makecell{Annual\\\\November--May\\February--March\\January--March\\Fall\\Summer\\Yearly\\January--March}\\
         \bottomrule
    \end{tabular}}
    \caption*{Table~\ref{southernhostdensities} continued.}
    %\label{ap:southernparamtable}
\end{table*}
\section*{Declaration}

\noindent \textbf{Competing Interests}\\
The authors declare that they have no known competing financial interests or personal relationships that could have appeared to influence the work reported in this paper.

\noindent \textbf{Funding}\\
This work was supported by the U.S. National Science Foundation under Grant No. DMS-2230790.

%-----Bibliography----------------------
\bibliographystyle{elsarticle-num}
\bibliography{refs.bib}

\end{document}